\documentclass[10pt,a4paper,hidelinks]{article}

\usepackage{LPZpreamble}

\usepackage{multicol}
\usepackage[normalem]{ulem} %

\makeatletter
\newcommand\mkcal@old[1]{%
  \expandafter\gdef\csname #1#1#1\endcsname{\ensuremath{\mathcal{#1}}}%
}
\forcsvlist{\mkcal@old}{A,B,C,D,E,F,G,H,I,J,K,L,M,N,O,P,Q,R,S,T,U,V,W,X,Y,Z}
\newcommand\mkscr@old[1]{%
  \expandafter\gdef\csname #1#1#1s\endcsname{\ensuremath{\mathscr{#1}}}%
}
\forcsvlist{\mkscr@old}{A,B,C,D,E,F,G,H,I,J,K,L,M,N,O,P,Q,R,S,T,U,V,W,X,Y,Z}
\makeatother

\renewcommand{\geq}{\geqslant}
\renewcommand{\leq}{\leqslant}
\renewcommand{\subset}{\subseteq}

\renewcommand{\tilde}[1]{\widetilde{#1}}
\renewcommand{\colon}{\mathbin{:}\ }

\newcommand{\C}{\ensuremath{\mathbb{C}}}

\newcommand{\Z}{\ensuremath{\mathbb{Z}}}

\newcommand{\id}{\ensuremath{\operatorname{id}}}

\newcommand{\pa}{\ensuremath{\partial}}

\renewcommand{\ker}{\ensuremath{\operatorname{ker}}}

\newcommand{\codim}{\operatorname{codim}}

\newcommand{\ord}[1]{\left\langle #1 \right\rangle}

\newcommand{\Hom}{\operatorname{Hom}}

\newcommand{\cone}{\operatorname{cone}}

\DeclarePairedDelimiterX\setc[2]{\{}{\}}{\,#1 \;\delimsize\vert\; #2\,}

\renewcommand{\CCC}{\mathscr{C}}

\newcommand{\Dcat}{\mathsf{D}}
\newcommand{\Dcatb}{\Dcat^{\mathrm{b}}}

\newcommand{\Obj}{\operatorname{Obj}}

\newcommand{\Rsf}{\mathsf{R}}

\newcommand{\sheaf}[1]{\mathscr{O}_{#1}}

\newcommand{\pro}{\mathbb{P}}
\newcommand{\Hlg}{\operatorname{H}}

\newcommand{\Ext}{\operatorname{Ext}}

\newcommand{\Perf}{\mathsf{D}^{\mathsf{perf}}}

\newcommand{\Pic}{\operatorname{Pic}}
\newcommand{\ideal}{\mathcal{I}} 
\newcommand{\Nor}[2]{\mathcal{N}_{#1 \mid #2}}

\newcommand{\dashto}{\dashrightarrow}

\newcommand{\Ku}{\operatorname{\mathcal{K}\mkern-1mu\mathit{u}}}

\newcommand{\Sym}{\operatorname{Sym}}

\newcommand{\dual}{{\raisebox{2pt}{\scalebox{0.7}[0.3]{$\bm\vee$}}}}

\newcommand{\floor}[1]{\left\lfloor #1 \right\rfloor}

\newcommand{\lmut}[1]{\mathbf{L}_{#1}}
\newcommand{\rmut}[1]{\mathbf{R}_{#1}}

\newcommand{\Bl}{\operatorname{Bl}}

\newcommand{\tY}{\widetilde{Y}}

\newcommand{\tT}{\widetilde{\cT}}

\newcommand{\perf}{\mathsf{perf}}

\newcommand{\RHom}{\operatorname{\mathsf{R}Hom}}

\newcommand{\Dperf}{\operatorname{\mathsf{D}}^{\mathsf{perf}}}

\newcommand{\ch}{\operatorname{ch}}

\title{Notes on the Kuznetsov Component of Cubic Sevenfolds}
\author{Peize Liu}
\date{}

\begin{document}

\maketitle

\begin{abstract}
    Let $Y$ be a cubic 7-fold.  When $Y$ is smooth, its Kuznetsov component $\Ku(Y)$ is a Calabi--Yau category of dimension $3$, and we construct a fully faithful embedding of $\Ku(Y)$ into the derived category of Clifford modules over $\pro^4$. When $Y$ is a general cubic 7-fold singular along a line, we construct an explicit weakly crepant categorical resolution of $\Ku(Y)$ by the derived category $\Dcatb(Z)$ of a smooth Calabi--Yau 3-fold $Z$, recovering a result of Favero--Kelly. In this construction, we express the resolution functor explicitly as a composition of geometric functors and mutations, and identify its kernel with a category equivalent to a pull-back of $\Dcatb(\pro^1,\sB_0)$, which may be viewed as a stacky curve.
\end{abstract}

\tableofcontents

\section{Introduction}

Let $Y$ be a cubic hypersurface in $\pro^8$. Its bounded derived category of coherent sheaves admits a semi-orthogonal decomposition (SOD):
\begin{equation}
    \Dcatb(Y) = \ord{\Ku(Y),\, \sheaf{Y},\, \sheaf{Y}(1),\, \sheaf{Y}(2),\, \sheaf{Y}(3),\, \sheaf{Y}(4),\, \sheaf{Y}(5)}, \label{SOD:cubic_7}
\end{equation}
where the full subcategory $\Ku(Y)$ with objects
\[\Obj(\Ku(Y)) = \left\{ E \in \Dcatb(Y) \mid \forall\, i \in \{0,\ldots,5\},\, \Ext^\bullet(\sheaf{Y}(i),E) = 0 \right\}\]
is called the \textbf{Kuznetsov component} of $Y$. 
In these notes we consider separately two cases: $Y$ is smooth, or has ordinary double points along a line.

\paragraph{Smooth cubic 7-fold.}

When $Y$ is smooth, $\Ku(Y)$ is a Calabi--Yau category of dimension $3$ (\cite[Corollary 4.3]{KuzV14}); it is considered a non-commutative Calabi--Yau 3-fold. In this case we prove the following result:

\begin{theorem}[][thm:emb_Ku_7]
    Let $Y$ be a smooth cubic $7$-fold, and let $\varPi_0 \subset Y$ be a $3$-plane. The quadric fibration $\Bl_{\varPi_0}Y \to \pro^{4}$ defines a sheaf of Clifford algebras, with the even part $\CCC_0$. Then the bounded derived category of $\CCC_0$-modules on $\pro^{4}$ admits the semi-orthogonal decomposition
    \[ \Dcatb(\pro^{4},\CCC_0) = \ord{\varPsi\sigma^*\Ku(Y),\, \CCC_1,\, \CCC_2,\, \CCC_3},\]
    where $\varPsi\circ\sigma^*\colon  \Ku(Y) \to \Dcatb(\pro^{4},\CCC_0)$ is a fully faithful functor.
\end{theorem}
This embedding result extends the known results for cubic 3-folds (\cite{BMMS}), cubic 4-folds (\cite{BLMS,kuz4fold}), and cubic 5-folds (\cite{LPZ25}).

\begin{conjecture}[][conj:smooth]
    Let $Y$ be a smooth cubic $7$-fold. Then $\Ku(Y)$ admits a Bridgeland stability condition.
\end{conjecture}
A potential strategy to the proof of \cref{conj:smooth}, following the method of \cite{BLMS}, would be to use the restriction of the rotated second tilt stability on $\Dcatb(\pro^{4},\CCC_0)$. This method depends on a conjectural $\ch_3$-inequality on the non-commutative $\pro^4$. The most hopeful approach is to adapt Chunyi's proof (\cite{Li26}) to non-commutative varieties.

As suggested by Alex Perry, another possible approach to proving the conjecture is to consider a particular type of singular cubic 7-folds, as introduced by \cite[Example 1.8]{FK18}. 

\paragraph{Cubic 7-folds singular along a line.}

\begin{definition}\label{def:sing_cubic}
  Let $X \subset \pro^8$ be a cubic 7-fold defined by the equation 
  \[f(x_0,\ldots,x_8) = x_7f_1(x_0,\ldots,x_6) + x_8f_2(x_0,\ldots,x_6) + g(x_0,\ldots,x_6),\]
  where $f_1, f_2 \in \C[x_0,...,x_6]$ are general homogeneous quadrics and $g \in \C[x_0,...,x_6]$ is a general homogeneous cubic. 
  Note that $X$ is singular along the line $\ell$ defined by $x_0 = \cdots = x_6 = 0$. We discuss its geometry in \cref{lem:sing_cubic_7}.
\end{definition}

\begin{definition}[][def:cat_res]
    Let $\cT$ be a DG-enhanced triangulated category. A \textbf{categorical resolution} of $\cT$ is a triple $(\tT,\pi_*,\pi^*)$, where
    \begin{enumerate}[nosep, before = \vspace{-\parskip}]
        \item $\tT$ is a smooth proper DG-enhanced triangulated category;
        \item $\pi_*\colon \tT \to \cT$ is left adjoint to $\pi^*\colon \cT^\perf \to \tT$:
              \[ \Hom_{\tT}(\pi^*E,F) \cong \Hom_{\cT}(E,\pi_*F), \quad E \in \cT^\perf,\ F \in \tT; \]
        \item The natural transformation $\iota \to \pi_*\pi^*$ is an isomorphism, where $\iota\colon \cT^\perf \hookrightarrow \cT$ is the inclusion functor.
    \end{enumerate}
    Here $\cT^\perf \subset \cT$ are the (left) homologically finite objects:
    \begin{equation}
        \Obj \cT^\perf \coloneqq \ab\{  E \in \cT \mid \forall\, F \in \cT,\ \RHom(E,F) \in \Dperf(\C)  \},
    \end{equation}
    For simplicity, we also call the functor $\pi_*\colon \tT \to \cT$ a categorical resolution when no confusion arises. 
    
    The categorical resolution $(\tT,\pi_*,\pi^*)$ is called \textbf{weakly crepant} if we also have $\pi_* \dashv \pi^*$ on $\cT^\perf$, i.e.
    \[ \Hom_{\tT}(F,\pi^*E) \cong \Hom_{\cT}(\pi_*F,E), \quad E \in \cT^\perf,\ F \in \tT. \]
\end{definition}

We recover and refine the following result of \cite[Example 1.8]{FK18}:

\begin{theorem}[][thm:sing]
    Let $Y$ be a singular cubic 7-fold as above. Then there is a weakly crepant categorical resolution $\sigma_*\circ\varTheta\colon \Dcatb(Z) \to \Ku(Y)$, where $Z$ is the smooth Calabi--Yau 3-fold defined as the $(2,2,3)$-complete intersection $\mathbb{V}(f_1,f_2,g) \subset \pro^6$. Moreover, the kernel of the resolution is given by  
    \begin{equation}
        \ker(\sigma_*\circ\varTheta) = i^*\Ku(W) \otimes \sheaf{Z}(1) \subset \Dcatb(Z),
    \end{equation}
    where $i\colon Z \hookrightarrow W$ is the inclusion of $Z$ into the $(2,2)$-complete intersection $W \subset \pro^6$ as an anti-canonical divisor, and $\Ku(W)$ is the Kuznetsov component of $W$ defined by 
    \begin{equation}
        \Dcatb(W) = \ord{\Ku(W), \ \sheaf{W},\, \sheaf{W}(1),\, \sheaf{W}(2)}.
    \end{equation}
\end{theorem}

Crepant categorical resolutions of isolated nodal singularities have been studied in \cite{kuz4fold, catt23, KS24}. It is known that the kernel of the resolution can be chosen to be generated by a single $2$-spherical object. In the case of 1-nodal cubic $4$-fold $Y_4$, \cite{kuz4fold} constructs the resolution $\Dcatb(S) \to \Ku(Y_4)$, where $S$ is a smooth K3 surface obtained as a $(2,3)$-complete intersection in $\pro^4$, and the kernel is generated by the pull-back of the spherical bundle on the defining quadric of $S$. In \cref{thm:sing}, the Kuznetsov component $\Ku(W)$ can be viewed as the category of relative spinor sheaves over the pencil of quadrics defined by $W$, thus providing a relative version of the result in \cite{catt23}. We also note that, as a category $\Ku(W) \simeq \Dcatb(\pro^1,\sB_0)$ is associated to a quadric fibration over $\pro^1$ branched in 7 points, and the category is considered a stacky curve.

In \cite[Example 1.8]{FK18}, the authors generalize the construction to higher-dimensional singular cubics. For $n = 3m - 2$, they consider a singular cubic $n$-fold $Y_{3m-2}$ which is nodal along a $(m-2)$-plane. They show that there exists a crepant categorical resolution $\Dcatb(Z_m) \to \Dcatb(Y_{3m-2})$ of the Kuznetsov component of $Y_{3m-2}$, where $Z_m$ is a smooth Calabi--Yau $m$-fold defined as a $(2^{m-1},3)$-complete intersection in $\pro^{2m}$. Their method relies on passing to Landau--Ginzburg models and applying variational GIT techniques. Our approach to \cref{thm:sing} instead takes the route of \cite{kuz4fold} and utilizes explicit mutation computations. As a result we are able to obtain an explicit description of the resolution functor and its kernel. It is not hard to generalize our construction to Calabi--Yau $m$-folds for any $m \geq 3$, in which the mutation computations are purely combinatorial.

As a consequence it is possible adapt the approach in \cite{LPZ_LCY,LPZ_nodalcubic5} to construct stability conditions on the singular cubic 7-fold $Y = Y^{\mathsf{sing}}$.

\begin{conjecture}[][conj:sing]
    Let $Y = Y^{\mathsf{sing}}$ be a singular cubic 7-fold as above. Then $\Ku(Y)$ admits a Bridgeland stability condition.
\end{conjecture}

Let us outline the strategy to prove \cref{conj:sing} and possibly \cref{conj:smooth}. We first construct a family of stability conditions on $\Dcatb(Z)$ using \cite{FKLR25}. By deforming the stability conditions to a suitable weak stability condition, the categorical resolution $\Dcatb(Z) \to \Ku(Y^{\mathsf{sing}})$ descends the stability conditions to $\Ku(Y^{\mathsf{sing}})$, using the tools from \cite{LPZ_LCY,LPZ_nodalcubic5}. For general smooth cubic $7$-fold $Y$, we construct a degeneration to $Y^{\mathsf{sing}}$, and the deformation techniques of \cite{LMPSZ} applies to show that the stability conditions on $\Ku(Y^{\mathsf{sing}})$ can be deformed to $\Ku(Y)$.

\paragraph{Relation to previous work.} After the first version of these notes appeared on arXiv, Zak Turčinović kindly brought to our attention his PhD thesis \cite{Turc18}. \cref{thm:sing} were obtained independently of \cite{Turc18}, but there is substantial overlap between the two works, both in the conclusions and in the methods of proof. More precisely, Tur\v{c}inovi\'{c} considered the complete intersection $X$ of multi-degree \vspace{-0.5em}
\begin{equation}
    (d+1,\underbrace{d,d,...,d}_{(k-1) \text{ times}})\vspace{-0.5em}
\end{equation}
in $\pro^{n-k+1}$ and constructed a weakly crepant categorical resolution by $\Ku(X)$ (or $\Dcatb(X)$ if $X$ is Calabi--Yau) of the Kuznetsov component $\Ku(Y_{d+1})$ of a hypersurface $Y_{d+1} \subset \pro^{n}$ of degree $(d+1)$ that vanishes to order $d$ along a linear subspace $\pro^{k-2} \subset Y$, where $2 \leq k \leq \floor{\frac{n+1}{d+1}}$. See \cite[Theorem 1.4 and Corollary 1.5]{Turc18} for the precise statements. The case of singular cubic 7-folds considered in these notes corresponds to the case $(d,n,k) = (2,8,3)$ of his construction. Nonetheless, our \cref{thm:sing} additionally establish the resolution as a Verdier localization and identify its kernel in terms of the Kuznetsov component of an intersection of two quadrics.

\subsection*{Acknowledgements}

I would like to thank Alex Perry for suggesting the result \cite[Example 1.8]{FK18}, and for hosting me during my visit to University of Michigan Ann Arbor in Winter 2026. I would also like to thank Zak Tur\v{c}inovi\'{c} for bringing his PhD thesis \cite{Turc18} to my attention after the first version of these notes appeared on arXiv, and for his kind correspondence.
This work is supported by the Warwick Mathematics Institute Centre for Doctoral Training, and I gratefully acknowledge funding from the University of Warwick. This work is also partially supported by the Royal Society URF{\textbackslash}R1{\textbackslash}201129 `Stability condition and application in algebraic geometry'.

\section{Smooth cubic 7-folds}\label{sec:sm_cubic_7}

Throughout this section, $Y$ is a smooth cubic 7-fold.

\subsection{Geometry of quadric fibrations}\label{subsec:cubic_7_quad}

\begin{setting}[][set:geom_quad_fib]
    Let $V$ be a $9$-dimensional vector space over $\C$, $A \subset V$ a $4$-dimensional subspace, and $B \coloneqq V/A$ the corresponding quotient space. Taking the projectivisation of the quotient map $V \twoheadrightarrow B$, we obtain a rational map $\pro^{8} = \pro(V) \dashto \pro(B) = \pro^{4}$, which is the projection away from the $3$-plane $\varPi_0 = \pro(A)$. The indeterminacy of this rational map is resolved by blowing up $\pro(V)$ along $\varPi_0$. Let $\tau\colon \Bl_{\varPi_0}\pro(V) \to \pro(V)$ be the blow-up and $q\colon \Bl_{\varPi_0}\pro(V) \to \pro(B)$ be the resulting morphism. Note that $q\colon \Bl_{\varPi_0}\pro(V) \to \pro(B)$ is a $\pro^{4}$-fibration, which makes $\Bl_{\varPi_0}\pro(V)$ the projective bundle $\pro_{\pro(B)}(\FFF)$, where
    \[\FFF \coloneqq \left(q_*\tau^*\sheaf{\pro(V)}(1) \right)^\dual \cong \sheaf{\pro(B)}^{\oplus 4} \oplus \sheaf{\pro(B)}(-1)\]
    is locally free of rank $5$. 
    
    The Fano variety $\FFF_3(Y)$ of $3$-planes in a smooth cubic 7-fold $Y$ has expected dimension 0, and a Chern class computation shows that it is non-empty (\cite[Lemma 6.6]{HP24}). That is, the smooth cubic hyperplane $Y \subset \pro(V)$ contains a $3$-plane $\varPi_0$. Denote by $\sigma\colon \tilde Y \to Y$ the embedded blow-up of $Y$ under $\tau\colon \pro(\FFF) \to \pro(V)$. Let $E \subset \tilde Y$ be the exceptional divisor. They are summarized in the following diagram:
\end{setting}

\begin{equation}\label{diag:cubic_7_quad_fib}
    \begin{tikzcd}[ampersand replacement=\&,cramped]
    \& E \& {\tY} \& {\pro_{\pro^4}(\sheaf{}^{\oplus 4} \oplus \sheaf{}(-1))} \\
    \varPi_0 \& Y \& {\pro^8} \&\& {\pro^4}
    \arrow["\iota", hook, from=1-2, to=1-3]
    \arrow["p"', from=1-2, to=2-1]
    \arrow["\alpha", hook, from=1-3, to=1-4]
    \arrow["\sigma"', from=1-3, to=2-2]
    \arrow["\pi", from=1-3, to=2-5]
    \arrow["\tau"', crossing over, from=1-4, to=2-3]
    \arrow["q", from=1-4, to=2-5]
    \arrow["j", hook, from=2-1, to=2-2]
    \arrow[hook, from=2-2, to=2-3]
\end{tikzcd}
\end{equation}

The following result is a generalization of \cites[Proposition 1.5.3]{huy23}[Lemma 5.1]{kuz4fold}.
\begin{lemma}[][]
    The morphism $\pi = q \circ \alpha\colon \tY \to \pro^4$ is a fibration in $3$-dimensional quadrics with the discriminant locus $D \in |\sheaf{\pro(B)}(7)|$. Let $H' \coloneqq \tau^*\sheaf{\pro^8}(1)$ and $h' \coloneqq q^*\sheaf{\pro^4}(1)$. Then $\tY = 2H' + h'$ in $\Pic \pro(\cF)$. Using $H$ and $h$ to denote the restriction of the classes $H'$ and $h'$ in $\tY$, then we have $E = H - h$ and $K_{\tY} = -3H - 3h$ in $\Pic \tY$.
\end{lemma}

Let $\cF \coloneqq \sheaf{}^{\oplus 4} \oplus \sheaf{}(-1)$ on $\pro^4$. The quadratic fibration $\pi\colon \tY \to \pro^4$ gives rise to a sheaf of Clifford algebras on $\pro^4$, with the even part $\CCC_0$ and the odd part $\CCC_1$ given by 
\begin{align*}
    \CCC_0 &\coloneqq \bigoplus_{m=0}^\infty {\textstyle\bigwedge^{2m}} \cF \otimes \sheaf{\pro^4}(-m) = \sheaf{\pro^4} \oplus \sheaf{\pro^4}(-1)^{\oplus 6} \oplus \sheaf{\pro^4}(-2)^{\oplus 5} \oplus \sheaf{\pro^4}(-3)^{\oplus 4}; \\
    \CCC_1 &\coloneqq \bigoplus_{m=0}^\infty {\textstyle\bigwedge^{2m+1}} \cF \otimes \sheaf{\pro^4}(-m) = \sheaf{\pro^4}^{\oplus 4} \oplus \sheaf{\pro^4}(-1)^{\oplus 5} \oplus \sheaf{\pro^4}(-2)^{\oplus 6} \oplus \sheaf{\pro^4}(-3).
\end{align*}
By convention we set $\CCC_{2\ell} \coloneqq \CCC_0 \otimes\sheaf{\pro^4}(\ell)$ and $\CCC_{2\ell+1} \coloneqq \CCC_1 \otimes\sheaf{\pro^4}(\ell)$ for any $\ell \in \Z$.

\subsection{Derived category of quadric fibrations}

\paragraph{Basic properties of mutations.} 

We recall the definition and some properties of mutation functors.

\begin{definition}
    Let $\cT$ be a triangulated category, and $\cD \subset \cT$ an admissible subcategory, with the inclusion functor $\beta_*\colon \cD \hookrightarrow \cT$. We have the adjunction triple $\beta^* \dashv \beta_* \dashv \beta^!$.
    \begin{enumerate}[nosep]
        \item $\lmut{\cD} \colon \cT \to \cD^{\perp}$, called the \textbf{left mutation} of $\cD$, gives the distinguished triangle in $\cT$:
            \begin{equation}\label{seq:def_left_mut}
                \begin{tikzcd}[ampersand replacement=\&]
                      {\beta_*\beta^!F} \& F \& {\lmut{\cD}F} \& {}
                      \arrow["{\operatorname{ev}}", from=1-1, to=1-2]
                      \arrow[from=1-2, to=1-3]
                      \arrow["{+1}", from=1-3, to=1-4]
                  \end{tikzcd} 
            \end{equation}
        \item $\rmut{\cD} \colon \cT \to {}^{\perp}\cD$, called the \textbf{right mutation} of $\cD$, gives the distinguished triangle in $\cT$:
              \[\begin{tikzcd}[ampersand replacement=\&]
                      {\rmut{\cD}F} \& F \& {\beta_*\beta^*F} \& {}
                      \arrow[from=1-1, to=1-2]
                      \arrow["{\operatorname{coev}}", from=1-2, to=1-3]
                      \arrow["{+1}", from=1-3, to=1-4]
                  \end{tikzcd}\]
    \end{enumerate}
    Strictly speaking, $\lmut{\cD}F$ and $\rmut{\cD}F$ should be embedded into $\cT$ via the corresponding inclusion functors.
\end{definition}
If $\cD = \ord{E}$ for $E \in \cT$ an exceptional object, the mutation functors $\lmut{E}$ and $\rmut{E}$ can be described explicitly as
\begin{align}
    \lmut{E}F &= \cone\ab( \begin{tikzcd}[ampersand replacement=\&, cramped]
                {\displaystyle\bigoplus_{k \in \Z}\Hom(E,F[k]) \otimes E[-k]} \& F
                \arrow["{\operatorname{ev}}", from=1-1, to=1-2]
            \end{tikzcd} )  \in \ord{E}^\perp; \\
    \rmut{E}F &= \cone\ab( \begin{tikzcd}[ampersand replacement=\&, cramped]
                F \& {\displaystyle\bigoplus_{k \in \Z}\Hom(F,E[k])^\dual \otimes E[k]}
                \arrow["{\operatorname{coev}}", from=1-1, to=1-2]
            \end{tikzcd} )[-1] \in {}^\perp\!\ord{E}.
\end{align}

\begin{lemma}[][lem:mutations_basic]
    Let $\cD \subset \cT$ be an admissible subcategory, and $F \in \cT$ any object.
    \begin{enumerate}[dense]
        \item If $\varphi$ is an auto-equivalence of $\cT$, then $\varphi \circ \lmut{\cD} \cong \lmut{\varphi(\cD)} \circ \varphi$, and $\varphi \circ \rmut{\cD} \cong \rmut{\varphi(\cD)} \circ \varphi$.

        \item $\lmut{\cD}|_{{}^{\perp}\cD}$ and $\rmut{\cD}|_{\cD^{\perp}}$ are mutually inverse. In particular, if $F \in \cD^{\perp}$, then $\lmut{\cD}F \cong F$ and $\lmut{\cD}\rmut{\cD}F \cong F$; and if $F \in {}^{\perp}\cD$, then $\rmut{\cD}F \cong F$ and $\rmut{\cD}\lmut{\cD}F \cong F$.

        \item Let $A,B \in \cT$ be exceptional objects such that $\RHom(A,B) = 0$. Then $\lmut{A} \circ \rmut{B} \cong \rmut{B} \circ \lmut{A}$.
    \end{enumerate}
\end{lemma}

\begin{lemma}[Mutation triangle][lem:mut_trig]
    Let $A,B,C \in \cT$ be exceptional objects. Assume that $\RHom(B,A) = 0$ and there exists a distinguished triangle:
    \begin{equation}
        \begin{tikzcd}[ampersand replacement=\&,cramped]
            A \& B \& C \& {A[1].}
            \arrow["f", from=1-1, to=1-2]
            \arrow["g", from=1-2, to=1-3]
            \arrow["{\delta}", from=1-3, to=1-4]
        \end{tikzcd} \label{trig:mut_lem}
    \end{equation}
    Then we have $\lmut{B}C \cong A[1]$ and $\rmut{A}C \cong B$.
\end{lemma}

\begin{lemma}[][lem:adj_triple]
    Let $\cT$ be a triangulated category with a Serre functor $\sfS$, and let $\cD \subset \cT$ be an admissible subcategory. Considering the left and right mutation functors as endofunctors on $\cT$, we have the infinite adjunction sequence:
    \begin{equation}
         \cdots \dashv \rmut{\sfS^{-1}\cD} \dashv \lmut{\sfS^{-1}\cD} \dashv \rmut{\cD} \dashv \lmut{\cD} \dashv \rmut{\sfS\cD} \dashv \lmut{\sfS\cD} \dashv \cdots
    \end{equation}
\end{lemma}

\paragraph{Semi-orthogonalities}

Note that $\Pic \tY = \Z H \oplus \Z h$. For ease of notation, we denote the sheaves $\sheaf{\tY}(aH+bh)$ and $\iota_*\sheaf{E}(aH+bh)$ by $\sheaf{}(a,b)$ and $\sheaf{E}(a,b)$ respectively, where $a,b \in \Z$. For $a=b=0$, these sheaves will be denoted simply by $\sheaf{}$ and $\sheaf{E}$.

Since $\sigma\colon \tY \to Y$ is a smooth blow-up, Orlov's formula (\cite[Theorem 2.6]{orlov92}) gives the semi-orthogonal decomposition
\begin{equation}
        \Dcatb(\tY)
         = \ord{\sigma^*\Dcatb(Y),\ \iota_*p^*\Dcatb(\varPi_0),\ \iota_*(p^*\Dcatb(\varPi_0) \otimes \sheaf{E}(-E)),\ \iota_*(p^*\Dcatb(\varPi_0) \otimes \sheaf{E}(-2E))}  
         \label{SOD:smooth_blowup_7}
    \end{equation}
On the other hand, since $\pi\colon \tY \to \pro^4$ is a quadric fibration, by \cite[Theorem 4.2]{kuzQuaFib} there is a semi-orthogonal decomposition:
\begin{align}
    \Dcatb(\tY)
     & = \ord{\varPhi\Dcatb(\pro^4,\CCC_0),\ \pi^*\Dcatb(\pro^4),\ \pi^*\Dcatb(\pro^4) \otimes \sheaf{\tilde{Y}}(H),\ \pi^*\Dcatb(\pro^4) \otimes \sheaf{\tilde{Y}}(2H)}    
     \label{SOD:quad_fib_7}
\end{align}
The fully faithful functor $\varPhi\colon \Dcatb(\pro^4,\CCC_0) \to \Dcatb(\tY)$ admits the left adjoint $\varPsi\colon \Dcatb(\tY) \to \Dcatb(\pro^4,\CCC_0)$ given by $F \longmapsto \pi_*(F \otimes \EEE \otimes \sheaf{\tY}(h))[3]$, where $\EEE$ is a right $\pi^*\CCC_0$-module that fits into the short exact sequence of right $q^*\CCC_0$-modules (\cite[Lemma 4.10]{kuzQuaFib}):
\begin{equation}
    \begin{tikzcd}[ampersand replacement=\&,cramped]
            0 \& {q^*\CCC_{-1}(-4H)} \& {q^*\CCC_0(-3H)} \& {\alpha_*\EEE} \& 0.
            \arrow[from=1-1, to=1-2]
            \arrow[from=1-2, to=1-3]
            \arrow[from=1-3, to=1-4]
            \arrow[from=1-4, to=1-5]
        \end{tikzcd} \label{ses:E_7}
\end{equation}
By \cite[Lemma 4.7]{kuzQuaFib}, the sheaves $\EEE$ and $\EEE'$ on $\tY$ are locally free of rank $8$ as $\sheaf{\tY}$-modules.

\begin{proposition}[Orthogonality relations for $\sheaf{\tY}$][prop:ortho_line_bun_7]
    The orthogonality relation $\Ext^\bullet(\sheaf{}(a,b), \sheaf{}) = 0$ holds on $\tY$ if:
    \begin{enumerate}[dense, label = \textnormal{(\roman*)}]
      \item either $a = 1$ or $2$;
      \item or $(a,b) \in \{ (-2,3),$ $(-1,2),$ $(-1,3),$ $(-1,4),$ $(0,1),$ $(0,2),$ $(0,3),$ $(0,4),$ $(3,-1),$ $(3,0),$ $(3,1),$ $(3,2),$ $(4,-1),$ $(4,0),$ $(4,1),$ $(5,0) \}$.
    \end{enumerate}
\end{proposition} 
\begin{proof}
    Let $\mathcal{S} = \left\{ (a,b) \in \Z^2 \mid \Ext^\bullet(\sheaf{}(a,b), \sheaf{}) = 0 \right\}$. 
    \begin{enumerate}[dense, label = \textnormal{(\roman*)}]
      \item For $a = 1$ or $2$, the fact that $\Ext^\bullet(\sheaf{}(a,b), \sheaf{}) = 0$ is already encoded in the SOD \eqref{SOD:quad_fib_7}. 
      Hence $(1,b), (2,b) \in \mathcal{S}$ for all $b \in \Z$.
      \item For $a = 0$ and $1 \leq b \leq 4$, we have 
      \[\Ext^\bullet(\sheaf{}(0,b), \sheaf{}) = \Hlg^\bullet(\pro^4, \sheaf{\pro^4}(-b)) = 0.\] 
      Hence $(0,b) \in \mathcal{S}$ for $1 \leq b \leq 4$.
      \item For $a = -1$ and $2 \leq b \leq 4$, we have 
      \[\Ext^\bullet(\sheaf{}(-1,b), \sheaf{}) = \Hlg^\bullet(\pro^4, \cF^\dual \otimes \sheaf{\pro^4}(-b)) = \Hlg^\bullet(\pro^4, \sheaf{\pro^4}(-b+1) \oplus \sheaf{\pro^4}(-b)^{\oplus 4}) = 0.\]
      Hence $(-1,b) \in \mathcal{S}$ for $2 \leq b \leq 4$.
      \item For $(a,b) = (5,0)$, we have 
      \[\Ext^\bullet(\sheaf{}(5,0), \sheaf{}) = \Hlg^\bullet(Y, \sheaf{Y}(-5)) = 0.\]
      Hence $(5,0) \in \mathcal{S}$.
      \item Finally, Since $\omega_{\tY} = \sheaf{}(-3,-3)$, by Serre duality, we have $\Ext^\bullet(\sheaf{}(a,b), \sheaf{}) \cong \Ext^{\bullet}(\sheaf{}(3-a,3-b),\ \sheaf{}[7])^\dual$. Hence $(a,b) \in \mathcal{S}$ if and only if $(3-a,3-b) \in \mathcal{S}$. In particular, this shows that $(-2,3)$, $(3,-1)$, $(3,0)$, $(3,1)$, $(3,2)$, $(4,-1)$, $(4,0)$, $(4,1) \in \mathcal{S}$. \qedhere
    \end{enumerate}
\end{proof}
\begin{remark}
  One can check that the sufficient conditions in \cref{prop:ortho_line_bun_7} are also necessary, by a lengthy but straightforward computation using Grothendieck--Riemann--Roch theorem on $\tY$.
\end{remark}

\begin{proposition}[Orthogonality relations for $\sheaf{E}$][prop:ortho_torsion_7]
  The orthogonality relation $\Ext^\bullet(\sheaf{}(a,b), \sheaf{E}) = 0$ holds on $\tY$ if one of the following holds:
    \begin{enumerate}[dense, label = \textnormal{(\roman*)}]
      \item $a = 1$, or
      \item $b = 1,2,3$, or
      \item $(a,b) \in \{(-1,4),\ (0,4),\ (2,0),\ (3,0) \}$.
    \end{enumerate}
  The orthogonality relation $\Ext^\bullet(\sheaf{E}, \sheaf{}(a,b)) = 0$ holds on $\tY$ if one the following holds:
    \begin{enumerate}[dense, label = \textnormal{(\roman*)}]
      \item $a = -2$, or
      \item $b = 0,-1,-2$, or
      \item $(a,b) \in \{(-4,1),\ (-3,1),\ (-1,-3),\ (0,-3)\}$.
    \end{enumerate}
\end{proposition}
\begin{proof}
  Since $E = H-h \in \Pic \tY$, the torsion sheaf $\sheaf{E}$ fits into the short exact sequence
  \begin{equation}\label{ses:cubic7_smooth_OE}
      \begin{tikzcd}[ampersand replacement=\&,cramped]
        0 \& {\sheaf{}(-1,1)} \& {\sheaf{}} \& {\sheaf{E}} \& 0.
        \arrow[from=1-1, to=1-2]
        \arrow[from=1-2, to=1-3]  
        \arrow[from=1-3, to=1-4]
        \arrow[from=1-4, to=1-5]
      \end{tikzcd}
  \end{equation}
  Hence the set
  \[\mathcal{T} \coloneqq \left\{ (a,b) \in \Z^2 \mid \Ext^\bullet(\sheaf{}(a,b),\sheaf{E}) = 0 \right\}\]
  contains the intersection of $\mathcal{S}$ in \cref{prop:ortho_line_bun_7} with $\mathcal{S} + (-1,1)$.
  Meanwhile, by \eqref{SOD:smooth_blowup_7}, we also have that 
  \begin{equation}
      \Ext^\bullet(\sheaf{}(a,b),\sheaf{E}) \cong \Ext^\bullet(\sheaf{E}(3,3),\sheaf{}(a,b))^{\dual}[7] = 0
  \end{equation}
  for $b = 1,2,3$. Hence these conditions together gives the first part of the statement of this proposition. 
  The second part is just Serre duality on $\tY$.
\end{proof}

\begin{remark}
  The left orthogonals of $\sheaf{}$ and $\sheaf{E}$ as stated \cref{prop:ortho_line_bun_7,prop:ortho_torsion_7} can be visualized on the plane representing $\Pic \tY$, where $H = \bm{e}_x$ and $h = \bm{e}_y$, as follows:
\end{remark}
\begin{figure}[H]
    \centering
    \begin{minipage}{0.48\textwidth}
        \centering
        \begin{tikzpicture}[
  x=0.8cm,y=0.8cm,
  every node/.style={font=\scriptsize},
  ortho point/.style={circle, draw=green!55!black, fill=white, line width=1.2pt, inner sep=1.5pt},
  ortho line/.style={green!55!black, line width=1.2pt},
  axis/.style={black, line width=.55pt, ->},
  grid line/.style={gray!35, line width=.25pt}
]
  \def\xmin{-2.6}\def\xmax{5.6}\def\ymin{-1.4}\def\ymax{4.75}

  \draw[step=1, grid line] (\xmin,\ymin) grid (\xmax,\ymax);

  \draw[ortho line] (1,\ymin) -- (1,\ymax);
  \draw[ortho line] (2,\ymin) -- (2,\ymax);

  \draw[axis] (\xmin,0) -- (\xmax+0.2,0) node[right] {$H$};
  \draw[axis] (0,\ymin) -- (0,\ymax+0.2) node[above] {$h$};

  \foreach \x in {-2,-1,3,4,5}
    {\draw (\x,.06)--(\x,-.06) node[below=2pt] {$\x$};}
    \foreach \x in {0,1,2}
    {\draw (\x,.06)--(\x,-.06) node[below=2pt, xshift = -5pt] {$\x$};}
  \foreach \y in {-1,1,2,3,4}
    {\draw (.06,\y)--(-.06,\y) node[left=2pt] {$\y$};}

  \foreach \y in {-1,0,1,2,3,4}{
    \node[ortho point] at (1,\y) {};
    \node[ortho point] at (2,\y) {};
  }

  \foreach \x/\y in {
    -2/3,
    -1/2,-1/3,-1/4,
     0/1,0/2,0/3,0/4,
     3/-1,3/0,3/1,3/2,
     4/-1,4/0,4/1,
     5/0
  }{
    \node[ortho point] at (\x,\y) {};
  }

  \node[green!55!black, above] at (1,\ymax) {$a=1$};
  \node[green!55!black, above] at (2,\ymax) {$a=2$};
\end{tikzpicture}
        \caption[Smooth cubic 7-fold: line bundles left orthogonal to $\sheaf{}$.]{Line bundles left orthogonal to $\sheaf{}$.}
        \label{fig:OY_7}
    \end{minipage}
    \begin{minipage}{0.48\textwidth}
        \centering
        \begin{tikzpicture}[
  x=0.8cm,y=0.8cm,
  every node/.style={font=\scriptsize},
  ortho point/.style={circle, draw=green!55!black, fill=white, line width=1.2pt, inner sep=1.5pt},
  ortho line/.style={green!55!black, line width=1.2pt},
  axis/.style={black, line width=.55pt, ->},
  grid line/.style={gray!35, line width=.25pt}
]
  \def\xmin{-2.6}\def\xmax{4.6}\def\ymin{-1.4}\def\ymax{4.75}

  \draw[step=1, grid line] (\xmin,\ymin) grid (\xmax,\ymax);

  \draw[ortho line] (1,\ymin) -- (1,\ymax);
    \draw[ortho line] (\xmin,1) -- (\xmax,1);
    \draw[ortho line] (\xmin,2) -- (\xmax,2);
    \draw[ortho line] (\xmin,3) -- (\xmax,3);

  \draw[axis] (\xmin,0) -- (\xmax+0.22,0) node[right] {$H$};
  \draw[axis] (0,\ymin) -- (0,\ymax+0.22) node[above] {$h$};

  \foreach \x/\xs in {-2/0,-1/0,0/-5,1/-5,2/0,3/0,4/0}{
    \draw (\x,.06)--(\x,-.06) node[below=2pt, xshift=\xs pt] {$\x$};
  }
  \foreach \y/\ys in {-1/0,1/-5,2/-5,3/-5,4/0}{
    \draw (.06,\y)--(-.06,\y) node[left=2pt, yshift=\ys pt] {$\y$};
  }

  \foreach \y in {-1,0,1,2,3,4}{
    \node[ortho point] at (1,\y) {};
  }
  \foreach \y in {1,2,3}{
    \foreach \x in {-2,-1,0,1,2,3,4}{
    \node[ortho point] at (\x,\y) {};
    }
  }

  \foreach \x/\y in {
    -2/3,
    -1/2,-1/3,-1/4,
     0/1,0/2,0/3,0/4,
     2/0,2/1,2/2,2/3,
     3/0,3/1,3/2,
     4/1
  }{
    \node[ortho point] at (\x,\y) {};
  }

  \node[green!55!black, above = 1pt, fill=white, inner sep=1pt]
    at (1,\ymax) {$a=1$};
    \foreach \y in {1,2,3}{
        \node[green!55!black, above = 1pt, fill=white, inner sep=1pt]
    at (\xmax,\y) {$b=\y$};
  }
    
\end{tikzpicture}
        \caption[Smooth cubic 7-fold: line bundles left orthogonal to $\sheaf{E}$.]{Line bundles left orthogonal to $\sheaf{E}$}
        \label{fig:OE_7}
    \end{minipage}
\end{figure}

Before going into the mutation process we state one more lemma.

\begin{lemma}[][lem:mut_E_7]
  We have $\lmut{\sheaf{}}\sheaf{E} \cong \sheaf{}(-1,1)[1]$ and $\rmut{\sheaf{}(-1,1)}\sheaf{E} \cong \sheaf{}$.
\end{lemma}
\begin{proof}
  Follows directly from the short exact sequence \eqref{ses:cubic7_smooth_OE}.
\end{proof}

\subsection[Decomposing ΦDᵇ(ℙ⁴,C₀)]{Decomposing $\varPhi\Dcatb(\pro^4,\CCC_0)$}

We will now perform a series of mutations on the SOD \eqref{SOD:smooth_blowup_7} of $\Dcatb(\tY)$ to obtain a SOD whose components match those in \eqref{SOD:quad_fib_7}. The mutation lemma \ref{lem:mut_E_7} allows us to `kill' all but two of the torsion sheaves in \eqref{SOD:smooth_blowup_7}; and the orthogonality relations in \cref{prop:ortho_line_bun_7,prop:ortho_torsion_7} allows us to permute the sheaves to match the form in \eqref{SOD:quad_fib_7}. We state the main result of this part.

\begin{proposition}[Decomposition of $\varPhi\Dcatb(\pro^4,\CCC_0)$][prop:SOD_Phi_7]
  The subcategory $\varPhi\Dcatb(\pro^4,\CCC_0)$ admits the semi-orthogonal decomposition
  \[\ord{\varphi_1\sigma^*\Ku(Y),\  \varphi_2\sheaf{}(4,0),\  \varphi_2\sheaf{E}(1,1),\  \varphi_3\sheaf{E}(4,1)}\]
  where $\varphi_1$, $\varphi_2$, $\varphi_3$ are the mutation functors defined as follows:
  \begin{align*}
    \varphi_1 &\coloneqq \lmut{\ord{\sheaf{}(0,-3),\ \sheaf{}(0,-2),\ \sheaf{}(0,-1),\ \sheaf{}(0,0),\ \sheaf{}(0,1),\ \sheaf{}(1,-1),\ \sheaf{}(1,0),\ \sheaf{}(2,-1)}}, \\
    \varphi_2 &\coloneqq \lmut{\pi^*\Dcatb(\pro^4) \otimes \sheaf{\tY}}
    \circ \lmut{\ord{\sheaf{}(1,-1),\ \sheaf{}(1,0),\ \sheaf{}(1,1),\ \sheaf{}(2,-1),\ \sheaf{}(2,0),\ \sheaf{}(2,1)}}
    \circ\lmut{\pi^*\Dcatb(\pro^4) \otimes \sheaf{\tY}(3H)}; \\ 
    \varphi_3 &\coloneqq \lmut{\pi^*\Dcatb(\pro^4) \otimes \sheaf{\tY}}
    \circ \lmut{\ord{\sheaf{}(1,-1),\ \sheaf{}(1,0),\ \sheaf{}(1,1),\ \sheaf{}(1,2),\ \sheaf{}(2,-1),\ \sheaf{}(2,0),\ \sheaf{}(2,1),\ \sheaf{}(2,2)}}
    \circ\lmut{\pi^*\Dcatb(\pro^4) \otimes \sheaf{\tY}(3H)}.
  \end{align*}
\end{proposition}
\begin{remark}
    The mutation functors $\varphi_1,\ \varphi_2,\ \varphi_3$ seem messy, but we will see in \cref{lem:psi_kill_7} that most parts of them are irrelevant. Effectively only $\lmut{\pi^*\Dcatb(\pro^4) \otimes \sheaf{\tY}(3H)}$ matters in later computations.
\end{remark}
\begin{proof}
  In this proof only, we use the following abbreviations:
  \[
    (a,b) \coloneqq \sheaf{}(a,b) = \sheaf{\tY}(aH+bh),\qquad [a,b] \coloneqq \sheaf{E}(a,b) = \sheaf{E}(aH+bh).
  \]
  \textbf{Step 0:} We start from the SOD \eqref{SOD:smooth_blowup_7}:
  \begin{equation}
      \Dcatb(\tY)
         = \ord{\sigma^*\Dcatb(Y),\ \iota_*p^*\Dcatb(\varPi_0),\ \iota_*(p^*\Dcatb(\varPi_0) \otimes \sheaf{E}(-E)),\ \iota_*(p^*\Dcatb(\varPi_0) \otimes \sheaf{E}(-2E))}. 
  \end{equation}
  For $\Dcatb(Y)$ we pull-back the following SOD:
  \begin{equation}
      \Dcatb(Y) = \ord{\sheaf{Y}(-1),\ \Ku(Y),\ \sheaf{Y},\ \sheaf{Y}(1),\ \sheaf{Y}(2),\ \sheaf{Y}(3),\ \sheaf{Y}(4)}.
  \end{equation}
  For the first copy of $\Dcatb(\varPi_0)$, we take the full exceptional collection $\ord{\sheaf{}(1),\sheaf{}(2),\sheaf{}(3),\sheaf{}(4)}$; for the second copy of $\Dcatb(\varPi_0)$ we take $\ord{\sheaf{}(2),\sheaf{}(3),\sheaf{}(4),\sheaf{}(5)}$; and for the third copy of $\Dcatb(\varPi_0)$ we take $\ord{\sheaf{}(4),\sheaf{}(5),\sheaf{}(6),\sheaf{}(7)}$. Since $E = H - h$, we have the expanded SOD for $\Dcatb(\tY)$:
  \begin{align}
        \Dcatb(\tY)
         &= \langle (-1,0),\ \sigma^*\Ku(Y),\ (0,0),\ (1,0),\ (2,0),\ (3,0),\ (4,0),\ [1,0],\ [2,0],\ [3,0],\ [4,0], \\
         &\qquad [1,1],\ [2,1],\ [3,1],\ [4,1],\ [2,2],\ [3,2],\ [4,2],\ [5,2] \rangle.
    \end{align}
  \textbf{Step 1:} Firstly left mutate $[1,0]$ through $\ord{(2,0),\ (3,0),\ (4,0)}$; secondly left mutate $[2,0]$ through $\ord{(3,0),\ (4,0)}$; and thirdly left mutate $[3,0]$ through $(4,0)$. By \cref{prop:ortho_torsion_7}, these permutations do not change the objects. Therefore we have:
    \begin{align}
        \Dcatb(\tY)
         &= \langle (-1,0),\ \sigma^*\Ku(Y),\ (0,0),\ (1,0),\ [1,0],\ (2,0),\ [2,0],\ (3,0),\ [3,0],\ (4,0),\ [4,0],  \\
         &\qquad [1,1],\ [2,1],\ [3,1],\ [4,1],\ [2,2],\ [3,2],\ [4,2],\ [5,2] \rangle.
    \end{align}
    \textbf{Step 2:} Simultaneously left mutate $[a,0]$ through $(a,0)$ for $a=1,2,3,4$, using \cref{lem:mut_E_7}. We obtain:
    \begin{align}
        \Dcatb(\tY)
         &= \langle (-1,0),\ \sigma^*\Ku(Y),\ (0,0),\ (0,1),\ (1,0),\ (1,1),\ (2,0),\ (2,1),\ (3,0),\ (3,1), \\
         &\qquad  (4,0),\ [1,1],\ [2,1],\ [3,1],\ [4,1],\ [2,2],\ [3,2],\ [4,2],\ [5,2]\rangle.
    \end{align}
    \textbf{Step 3:} Right mutate $\sigma^*\Ku(Y)$ through $\cC_1 \coloneqq \ord{(0,0),\ (0,1),\ (1,0)}$, and left mutate $\ord{(4,0),\ [1,1]}$ through $\cC_2 \coloneqq \langle(1,1),$ $(2,0),$ $(2,1),$ $(3,0),$ $(3,1)\rangle$. We obtain:
    \begin{align}
        \Dcatb(\tY)
         &= \langle (-1,0),\ (0,0),\ (0,1),\ (1,0),\ \rmut{\cC_1}\sigma^*\Ku(Y),\ \lmut{\cC_2}(4,0),\ \lmut{\cC_2}[1,1],\ (1,1), \\
         &\qquad (2,0),\ (2,1),\ (3,0),\ (3,1),\ [2,1],\ [3,1],\ [4,1],\ [2,2],\ [3,2],\ [4,2],\ [5,2]\rangle.
    \end{align}
    \textbf{Step 4:} Mutate $\ord{(-1,0),\ (0,0),\ (0,1),\ (1,0)}$ through its left orthogonal, using the Serre functor $\mathsf S = (- \otimes (-3,-3))[7]$. We obtain:
    \begin{align}
        \Dcatb(\tY)
         &= \langle \rmut{\cC_1}\sigma^*\Ku(Y),\ \lmut{\cC_2}(4,0),\ \lmut{\cC_2}[1,1],\ (1,1),\ (2,0),\ (2,1),\ (3,0),\ (3,1), \\
         &\qquad [2,1],\ [3,1],\ [4,1],\ [2,2],\ [3,2],\ [4,2],\ [5,2],\ (2,3),\ (3,3),\ (3,4),\ (4,3)\rangle
    \end{align}
    \textbf{Step 5:} Left mutate $[2,1]$ through $\ord{(3,0),\ (3,1)}$. By \cref{prop:ortho_torsion_7}, the permutation do not change the object. Therefore we have:
    \begin{align}
        \Dcatb(\tY)
         &= \langle \rmut{\cC_1}\sigma^*\Ku(Y),\ \lmut{\cC_2}(4,0),\ \lmut{\cC_2}[1,1],\ (1,1),\ (2,0),\ (2,1),\ [2,1],\ (3,0), \\
         &\qquad (3,1),\ [3,1],\ [4,1],\ [2,2],\ [3,2],\ [4,2],\ [5,2],\ (2,3),\ (3,3),\ (3,4),\ (4,3)\rangle.
    \end{align}
    \textbf{Step 6:} Simultaneously left mutate $[a,1]$ through $(a,1)$ for $a=2,3$, using \cref{lem:mut_E_7}. We obtain:
    \begin{align}
        \Dcatb(\tY)
         &= \langle \rmut{\cC_1}\sigma^*\Ku(Y),\ \lmut{\cC_2}(4,0),\ \lmut{\cC_2}[1,1],\ (1,1),\ (2,0),\ (1,2),\ (2,1),\ (3,0),\\
         &\qquad (2,2),\ (3,1),\ [4,1],\ [2,2],\ [3,2],\ [4,2],\ [5,2],\ (2,3),\ (3,3),\ (3,4),\ (4,3)\rangle.
    \end{align}
    \textbf{Step 7:} (i) right mutate $[5,2]$ through $\ord{(2,3),\ (3,3),\ (3,4)}$; and (ii) right mutate $[4,2]$ through $(2,3)$. By \cref{lem:mut_E_7}, these permutations do not change the objects. Therefore we have:
    \begin{align}
        \Dcatb(\tY)
         &= \langle \rmut{\cC_1}\sigma^*\Ku(Y),\ \lmut{\cC_2}(4,0),\ \lmut{\cC_2}[1,1],\ (1,1),\ (2,0),\ (1,2),\ (2,1),\ (3,0),\\
         &\qquad (2,2),\ (3,1),\ [4,1],\ [2,2],\ [3,2],\ (2,3),\ [4,2],\ (3,3),\ (3,4),\ [5,2],\ (4,3)\rangle.
    \end{align}
    \textbf{Step 8:} Simultaneously right mutate $[a,2]$ through $(a-1,3)$ for $a=3,4,5$, using \cref{lem:mut_E_7}. We obtain:
    \begin{align}
        \Dcatb(\tY)
         &= \langle \rmut{\cC_1}\sigma^*\Ku(Y),\ \lmut{\cC_2}(4,0),\ \lmut{\cC_2}[1,1],\ (1,1),\ (2,0),\ (1,2),\ (2,1),\ (3,0),\\
         &\qquad (2,2),\ (3,1),\ [4,1],\ [2,2],\ (2,3),\ (3,2),\ (3,3),\ (4,2),\ (3,4),\ (4,3),\ (5,2)\rangle.
    \end{align}
    \textbf{Step 9:} Left mutate $[4,1]$ through $\cC_3 \coloneqq \ord{(1,1),\ (2,0),\ (1,2),\ (2,1),\ (3,0),\ (2,2),\ (3,1)}$. We obtain:
    \begin{align}
        \Dcatb(\tY)
         &= \langle \rmut{\cC_1}\sigma^*\Ku(Y),\ \lmut{\cC_2}(4,0),\ \lmut{\cC_2}[1,1],\ \lmut{\cC_3}[4,1],\ (1,1),\ (2,0),\ (1,2),\ (2,1),\\
         &\qquad (3,0),\ (2,2),\ (3,1),\ [2,2],\ (2,3),\ (3,2),\ (3,3),\ (4,2),\ (3,4),\ (4,3),\ (5,2)\rangle.
    \end{align}
    \textbf{Step 10:} Left mutate $[2,2]$ first through $(3,1)$, which does not change the object by orthogonality, and then through $(2,2)$. Using \cref{lem:mut_E_7} we have:
    \begin{align}
        \Dcatb(\tY)
         &= \langle \rmut{\cC_1}\sigma^*\Ku(Y),\ \lmut{\cC_2}(4,0),\ \lmut{\cC_2}[1,1],\ \lmut{\cC_3}[4,1],\ (1,1),\ (2,0),\ (1,2),\ (2,1),\\
         &\qquad (3,0),\ (1,3),\ (2,2),\ (3,1),\ (2,3),\ (3,2),\ (3,3),\ (4,2),\ (3,4),\ (4,3),\ (5,2)\rangle.
    \end{align}
    \textbf{Step 11:} (i) Left mutate $(1,2)$ through $(2,0)$; (ii) left mutate $(1,3)$ through $\ord{(2,0),\ (2,1),\ (3,0)}$; (iii) left mutate $(2,2)$ through $(3,0)$; (iv) left mutate $(2,3)$ through $\ord{(3,0),\ (3,1)}$; and (v) left mutate $(3,4)$ through $(4,2)$. These permutations does not change any objects by the orthogonality relations in \cref{prop:ortho_line_bun_7}. We have:
    \begin{align}
        \Dcatb(\tY)
         &= \langle \rmut{\cC_1}\sigma^*\Ku(Y),\ \lmut{\cC_2}(4,0),\ \lmut{\cC_2}[1,1],\ \lmut{\cC_3}[4,1],\ (1,1),\ (1,2),\ (1,3),\ (2,0),\\
         &\qquad (2,1),\ (2,2),\ (2,3),\ (3,0),\ (3,1),\ (3,2),\ (3,3),\ (3,4),\ (4,2),\ (4,3),\ (5,2)\rangle. 
    \end{align}
    \textbf{Step 12:} Mutate $\ord{(3,0),\ (3,1),\ (3,2),\ (3,3),\ (3,4),\ (4,2),\ (4,3),\ (5,2)}$ through its right orthogonal, using the Serre functor $\mathsf S = (- \otimes (-3,-3))[7]$. We obtain:
    \begin{align}
        \Dcatb(\tY)
         &= \langle (0,-3),\ (0,-2),\ (0,-1),\ (0,0),\ (0,1),\ (1,-1),\ (1,0),\ (2,-1),\\
         &\qquad \rmut{\cC_1}\sigma^*\Ku(Y),\ \lmut{\cC_2}(4,0),\ \lmut{\cC_2}[1,1],\ \lmut{\cC_3}[4,1],\\
         &\qquad (1,1),\ (1,2),\ (1,3),\ (2,0),\ (2,1),\ (2,2),\ (2,3)\rangle. 
    \end{align}
    \textbf{Step 13:} Left mutate $\ord{\rmut{\cC_1}\sigma^*\Ku(Y),\ \lmut{\cC_2}(4,0),\ \lmut{\cC_2}[1,1],\ \lmut{\cC_3}[4,1]}$ through\\ $\cC_4 \coloneqq \ord{(0,-3),\ (0,-2),\ (0,-1),\ (0,0),\ (0,1),\ (1,-1),\ (1,0),\ (2,-1)}$. We have:
    \begin{align}
        \Dcatb(\tY)
         &= \langle \lmut{\cC_4}\rmut{\cC_1}\sigma^*\Ku(Y),\ \lmut{\cC_4}\lmut{\cC_2}(4,0),\ \lmut{\cC_4}\lmut{\cC_2}[1,1],\ \lmut{\cC_4}\lmut{\cC_3}[4,1],  \\
         &\qquad (0,-3),\ (0,-2),\ (0,-1),\ (0,0),\ (0,1),\ (1,-1),\ (1,0),\\
         &\qquad (2,-1),\ (1,1),\ (1,2),\ (1,3),\ (2,0),\ (2,1),\ (2,2),\ (2,3)\rangle.
    \end{align}
    \textbf{Step 14:} Right mutate $(2,-1)$ through $\ord{(1,1),\ (1,2),\ (1,3)}$, which is completely orthogonal to it by \cref{prop:ortho_line_bun_7}. We arrive at the form:
    \begin{align}
        \Dcatb(\tY)
         &= \langle \lmut{\cC_4}\rmut{\cC_1}\sigma^*\Ku(Y),\ \lmut{\cC_4}\lmut{\cC_2}(4,0),\ \lmut{\cC_4}\lmut{\cC_2}[1,1],\ \lmut{\cC_4}\lmut{\cC_3}[4,1],  \\
         &\qquad (0,-3),\ (0,-2),\ (0,-1),\ (0,0),\ (0,1),\ (1,-1),\ (1,0),\ (1,1),\ (1,2),\ (1,3),\\
         &\qquad (2,-1),\ (2,0),\ (2,1),\ (2,2),\ (2,3)\rangle
    \end{align}
    \textbf{Step 15:} We make the identifications that 
    \begin{align}
        \ord{(0,-3),\ (0,-2),\ (0,-1),\ (0,0),\ (0,1)} &= \pi^*\Dcatb(\pro^4), \\
        \ord{(1,-1),\ (1,0),\ (1,1),\ (1,2),\ (1,3)} &= \pi^*\Dcatb(\pro^4) \otimes \sheaf{\tY}(H), \\ 
        \ord{(2,-1),\ (2,0),\ (2,1),\ (2,2),\ (2,3)} &= \pi^*\Dcatb(\pro^4) \otimes \sheaf{\tY}(2H).
    \end{align}
    Then we have:
    \begin{align}
        \Dcatb(\tY)
         = \langle \lmut{\cC_4}\rmut{\cC_1}\sigma^*\Ku(Y),\ \lmut{\cC_4}\lmut{\cC_2}(4,0),\ \lmut{\cC_4}\lmut{\cC_2}[1,1],\ \lmut{\cC_4}\lmut{\cC_3}[4,1],&  \\
         \pi^*\Dcatb(\pro^4),\ \pi^*\Dcatb(\pro^4) \otimes \sheaf{\tY}(H),\ \pi^*\Dcatb(\pro^4) \otimes \sheaf{\tY}(2H)\rangle.&
    \end{align}
    By comparing the above SOD with the SOD \eqref{SOD:quad_fib_7}, we deduce that $\varPhi\Dcatb(\pro^4,\CCC_0)$ is equivalent to 
    \begin{equation}\label{equ:A.1.8_unsimplified}
    \begin{aligned}
        &\ord{\pi^*\Dcatb(\pro^4),\ \pi^*\Dcatb(\pro^4) \otimes \sheaf{\tY}(H),\ \pi^*\Dcatb(\pro^4) \otimes \sheaf{\tY}(2H)}^\perp \\ 
        &\cong \ord{\lmut{\cC_4}\rmut{\cC_1}\sigma^*\Ku(Y),\ \lmut{\cC_4}\lmut{\cC_2}(4,0),\ \lmut{\cC_4}\lmut{\cC_2}[1,1],\ \lmut{\cC_4}\lmut{\cC_3}[4,1]}.
    \end{aligned}
    \end{equation}
    \textbf{Step 16.} We can further simplify the mutations functors in \eqref{equ:A.1.8_unsimplified}. Recall that 
    \begin{align}
        \cC_1 &= \ord{(0,0),\ (0,1),\ (1,0)}, \\ 
        \cC_4 &= \ord{(0,-3),\ (0,-2),\ (0,-1),\ (0,0),\ (0,1),\ (1,-1),\ (1,0),\ (2,-1)}.
    \end{align}
    Note that $(1,0) \in \ord{(2,-1)}^{\perp}$ and $(0,0),\ (0,1) \in \ord{(1,-1),\ (1,0),\ (2,-1)}^{\perp}$. By \cref{lem:mutations_basic}.(3) and the fact that $\lmut{E}\rmut{E} = \lmut{E}$ for any object $E$, we have 
    \begin{align}
        \lmut{\cC_4}\rmut{\cC_1} &= \lmut{\ord{(0,-3),\ (0,-2),\ (0,-1)}}\lmut{(0,0)}\rmut{(0,0)}\lmut{(0,1)}\rmut{(0,1)}\lmut{(1,-1)}\lmut{(1,0)}\rmut{(1,0)}\lmut{(2,-1)} \\ 
        &= \lmut{\ord{(0,-3),\ (0,-2),\ (0,-1),\ (0,0),\ (0,1),\ (1,-1),\ (1,0),\ (2,-1)}}  \\ 
        &= \varphi_1.
    \end{align}
    For the mutation functor $\lmut{\cC_4}\lmut{\cC_2}$, recall that 
    \begin{equation}
        \cC_2 = \ord{(1,1),\ (2,0),\ (2,1),\ (3,0),\ (3,1)}.
    \end{equation}
    Note that $(4,0),\ [1,1] \in \ord{(3,2),\ (3,3),\ (3,4)}^{\perp}$, we have 
    \begin{equation}
    \begin{aligned}
        &\lmut{\cC_4}\lmut{\cC_2}\ord{(4,0),\ [1,1]} \\
        &= 
        \lmut{\cC_4}\lmut{\cC_2}\lmut{\ord{(3,2),\ (3,3),\ (3,4)}}\ord{(4,0),\ [1,1]} \\ 
        &= \lmut{\pi^*\Dcatb(\pro^4) \otimes \sheaf{\tY}}
    \circ \lmut{\ord{(1,-1),\ (1,0),\ (1,1),\ (2,-1),\ (2,0),\ (2,1)}}
    \circ\lmut{\pi^*\Dcatb(\pro^4) \otimes \sheaf{\tY}(3H)}\ord{(4,0),\ [1,1]} \\
        &= \varphi_2\ord{(4,0),\ [1,1]}.
    \end{aligned}
    \end{equation}
    For the mutation functor $\lmut{\cC_4}\lmut{\cC_3}$, recall that 
    \begin{equation}
        \cC_3 = \ord{(1,1),\ (2,0),\ (1,2),\ (2,1),\ (3,0),\ (2,2),\ (3,1)}.
    \end{equation}
     Similarly we have $[4,1] \in \ord{(3,2),\ (3,3),\ (3,4)}^{\perp}$. Therefore
    \begin{align}
        &\lmut{\cC_4}\lmut{\cC_3}[4,1] \\ 
        &= \lmut{\cC_4}\lmut{\cC_3}\lmut{\ord{(3,2),\ (3,3),\ (3,4)}}[4,1] \\ 
        &= \lmut{\pi^*\Dcatb(\pro^4) \otimes \sheaf{\tY}}
    \circ \lmut{\ord{(1,-1),\ (1,0),\ (1,1),\ (1,2),\ (2,-1),\ (2,0),\ (2,1),\ (2,2)}}
    \circ\lmut{\pi^*\Dcatb(\pro^4) \otimes \sheaf{\tY}(3H)}[4,1] \\
        &= \varphi_3[4,1]. \qedhere
    \end{align}
\end{proof} 

\subsection[Embedding of Ku(Y) into Dᵇ(ℙ⁴,C₀)]{Embedding of $\Ku(Y)$ into $\Dcatb(\pro^4,\CCC_0)$}

Next we apply the functor $\varPsi$ on the SOD of $\varPhi\Dcatb(\pro^4,\CCC_0)$ to obtain a semi-orthogonal decomposition of $\Dcatb(\pro^4,\CCC_0)$.

\begin{lemma}[][lem:psi_kill_7]
    For $a \in \{0,1,2\}$ and $b \in \Z$, $\varPsi(\sheaf{}(a,b)) = 0$. In particular, we have isomorphisms of functors:\vspace{-\parskip}
    \[\varPsi \circ \lmut{\sheaf{}(a,b)} \cong  \varPsi \cong \varPsi \circ \rmut{\sheaf{}(a,b)}.\]
\end{lemma}
\begin{proof}
    Note that $\varPsi(\sheaf{}(a,b)) = \pi_*(\EEE \otimes \sheaf{\tY}(aH + (b+1)h))[3]$. We twist the sequence \eqref{ses:E_7} by $\sheaf{}(a, b+1)$ and push forward along $q\colon \pro(\cF) \to \pro^4$, obtaining the distinguished triangle:
    \[\begin{tikzcd}[cramped, sep = scriptsize]
            {\CCC_{-1}(b+1) \otimes q_*\sheaf{\pro(\cF)}((a-4)H')} & {\CCC_{0}(b+1) \otimes q_*\sheaf{\pro(\cF)}((a-3)H')} & {\varPsi(\sheaf{\tY}(a,b))[-3]} & {}
            \arrow[from=1-1, to=1-2]
            \arrow[from=1-2, to=1-3]
            \arrow["{+1}", from=1-3, to=1-4]
        \end{tikzcd}\]
    For $a=0,1,2$, since $q$ is a $\pro^4$-fibration, we have $q_*\sheaf{\pro(\cF)}((a-3)H') = q_*\sheaf{\pro(\cF)}((a-4)H') = 0$. Hence $\varPsi(\sheaf{}(a,b)) = 0$.

    For any $F \in \Dcatb(\tY)$, the left mutation $\lmut{\sheaf{}(a,b)}$ followed by $\varPsi$ gives the distinguished triangle
    \[\begin{tikzcd}
            {\displaystyle\bigoplus_{k \in \Z}\Hom(\sheaf{}(a,b),F[k]) \otimes \varPsi(\sheaf{}(a,b))[-k]} & {\varPsi(F)} & {\varPsi \circ \lmut{\sheaf{}(a,b)}F} & {}
            \arrow[from=1-1, to=1-2]
            \arrow[from=1-2, to=1-3]
            \arrow["{+1}", from=1-3, to=1-4]
        \end{tikzcd}\]
    Hence there is a functorial isomorphism $\varPsi(F) \cong \varPsi \circ \lmut{\sheaf{}(a,b)}F$. That is, $\varPsi = \varPsi \circ \lmut{\sheaf{}(a,b)}$. Similarly we also have $\varPsi = \varPsi \circ \rmut{\sheaf{}(a,b)}$.
\end{proof}

The above lemma has an immediate consequence on the semi-orthogonal decomposition obtained in \cref{prop:SOD_Phi_7}
\begin{corollary}[][cor:Psi_on_mut_7]
  Let $\beta_*\colon \Dcatb(\pro^4) \to \Dcatb(\tY)$ be the fully faithful functor given by $F \longmapsto \pi^*F \otimes \sheaf{\tY}(3H)$. We have isomorphisms of functors:
  \begin{enumerate}[dense]
      \item $\varPsi\circ\varphi_1 \cong \varPsi$;
      \item $\varPsi\circ\varphi_2 \cong \varPsi\circ\varphi_3 \cong\varPsi \circ \lmut{\beta_*\Dcatb(\pro^4)}$.
    \end{enumerate}
\end{corollary}

\begin{lemma}[][lem:lmut_B_7]
    For $F \in \Dcatb(\tY)$, the left mutation $\lmut{\beta_*\Dcatb(\pro^4)}F$ of $F$ fits into the distinguished triangle:
    \[\begin{tikzcd}
            {\pi^*\pi_*(F \otimes \sheaf{\tY}(-3H)) \otimes \sheaf{\tY}(3H)} & F & {\lmut{\beta_*\Dcatb(\pro^4)}F} & {}
            \arrow[from=1-1, to=1-2]
            \arrow[from=1-2, to=1-3]
            \arrow["{+1}", from=1-3, to=1-4]
        \end{tikzcd}\]
\end{lemma}
\begin{proof}
    The functor $\beta_*$ admits the right adjoint $\beta^!\colon \Dcatb(\tY) \to \Dcatb(\pro^4)$, $F \longmapsto \pi_*(F \otimes \sheaf{\tY}(-3H))$. By definition, the left mutation of $F \in \Dcatb(\tY)$ fits into the distinguished triangle
    \[\begin{tikzcd}
            {\beta_*\beta^!F} & F & {\lmut{\beta_*\Dcatb(\pro^4)}F} & {}
            \arrow[from=1-1, to=1-2]
            \arrow[from=1-2, to=1-3]
            \arrow["{+1}", from=1-3, to=1-4]
        \end{tikzcd} \qedhere\]
\end{proof}

\begin{lemma}[][lem:C1_7] 
    \begin{enumerate}[dense]
    \item $\varPsi(\sheaf{\tY}(3H)) \cong \CCC_2[3].$
        \item $\varPsi\circ\lmut{\beta_*\Dcatb(\pro^4)}\sheaf{\tY}(4H) \cong \CCC_1[4].$
        \item $\varPsi\circ\lmut{\beta_*\Dcatb(\pro^4)}\sheaf{E}(H+h) = \CCC_2[2].$
        \item $\varPsi\circ\lmut{\beta_*\Dcatb(\pro^4)}\sheaf{E}(4H+h) = \CCC_3[4].$
    \end{enumerate}
\end{lemma}
\begin{proof}
    \begin{enumerate}
        \item $\varPsi(\sheaf{\tY}(3H)) = \pi_*(\EEE \otimes \sheaf{\tY}(3H+h))[3]$ fits into the distinguished triangle:
        \[\begin{tikzcd}[row sep=tiny]
            {\CCC_{-1} \otimes q_*\sheaf{\tY}(-H+h)[3]} & {\CCC_0 \otimes q_*\sheaf{\tY}(h)[3]} & {\pi_*(\EEE \otimes \sheaf{\tY}(3H+h))[3]} & {} 
            \arrow[from=1-1, to=1-2]
            \arrow[from=1-2, to=1-3]
            \arrow["{+1}", from=1-3, to=1-4]
        \end{tikzcd}\]
        Since $q_*\sheaf{\tY}(-H) = 0$, we have 
        \begin{equation}
            \varPsi(\sheaf{\tY}(3H)) = \pi_*(\EEE \otimes \sheaf{\tY}(3H+h))[3] \cong \CCC_0 \otimes q_*\sheaf{\tY}(h)[3] \cong \CCC_2[3].
        \end{equation}

        \item By \cref{lem:lmut_B_7}, $\varPsi\circ\lmut{\beta_*\Dcatb(\pro^4)}\sheaf{\tY}(4H)$ fits into the distinguished triangle:
        \begin{equation}
            \begin{tikzcd}
                    {\varPsi(\pi^*\pi_*\sheaf{\tY}(H) \otimes \sheaf{\tY}(3H))} & \varPsi(\sheaf{\tY}(4H)) & {\varPsi\circ\lmut{\beta_*\Dcatb(\pro^4)}\sheaf{\tY}(4H)} & {}
                    \arrow[from=1-1, to=1-2]
                    \arrow[from=1-2, to=1-3]
                    \arrow["{+1}", from=1-3, to=1-4]
                \end{tikzcd} \label{ses:H_7}
        \end{equation}
        Note that $\pi^*\pi_*\sheaf{\tY}(H) \otimes \sheaf{\tY}(3H) = \pi^*\cF^\dual \otimes \sheaf{\tY}(3H)$. By projection formula and (1), we have 
        \begin{align}
            \varPsi(\pi^*\pi_*\sheaf{\tY}(H) \otimes \sheaf{\tY}(3H)) 
            &= \pi_*(\pi^*\cF^\dual \otimes \EEE \otimes \sheaf{\tY}(3H+h))[3] \\
            &= \cF^\dual \otimes \pi_*(\EEE \otimes \sheaf{\tY}(3H+h))[3] \\ 
            &= q_*\sheaf{\tY}(H) \otimes \CCC_2[3] = q_*\sheaf{\tY}(H+h) \otimes \CCC_0[3].
        \end{align}
        Meanwhile, $\varPsi(\sheaf{\tY}(4H)) = \pi_*(\EEE \otimes \sheaf{\tY}(4H+h))[3]$ fits into the distinguished triangle:
        \[\begin{tikzcd}[row sep=tiny]
                {\CCC_{-1} \otimes q_*\sheaf{\tY}(h)[3]} & {\CCC_0 \otimes q_*\sheaf{\tY}(H+h)[3]} & {\pi_*(\EEE \otimes \sheaf{\tY}(4H+h))[3]} & {} \\
                {\CCC_1[3]} & {\varPsi(\pi^*\pi_*\sheaf{\tY}(H) \otimes \sheaf{\tY}(3H))} & {\varPsi(\sheaf{\tY}(4H))}
                \arrow[from=1-1, to=1-2]
                \arrow[from=1-2, to=1-3]
                \arrow[equal, from=1-1, to=2-1]
                \arrow[equal, from=1-2, to=2-2]
                \arrow["{+1}", from=1-3, to=1-4]
                \arrow[equal, from=1-3, to=2-3]
            \end{tikzcd}\]
        Comparing this with the sequence \eqref{ses:H_7}, we deduce that 
        \begin{equation}
            \varPsi\circ\lmut{\beta_*\Dcatb(\pro^4)}\sheaf{\tY}(4H) \cong (\CCC_{1}[3])[1] = \CCC_1[4].
        \end{equation}

        \item By \cref{lem:lmut_B_7}, $\lmut{\beta_*\Dcatb(\pro^4)}\sheaf{E}(H+h)$ fits into the distinguished triangle
        \[\begin{tikzcd}
                {\pi^*\pi_*\sheaf{E}(-2H+h) \otimes \sheaf{\tY}(3H)} & \sheaf{E}(H+h) & {\lmut{\beta_*\Dcatb(\pro^4)}\sheaf{E}(H+h)} & {}
                \arrow[from=1-1, to=1-2]
                \arrow[from=1-2, to=1-3]
                \arrow["{+1}", from=1-3, to=1-4]
            \end{tikzcd}\]
        We claim that $\pi_*\sheaf{E}(-2H+h) \cong \sheaf{\pro^4}[-2]$. Twisting the sequence \eqref{ses:cubic7_smooth_OE} by $\sheaf{}(-2H+h)$ and pushing forward along $\pi\colon \tY \to \pro^4$, since $\pi_*\sheaf{\tY}(-2H) = 0$, we have $\pi_*\sheaf{E}(-2H+h) = \pi_*\sheaf{\tY}(-3H + 2h)[1]$. Next, using that $\tY = 2H' + h'$ in $\Pic \pro(\cF)$, we have the distinguished triangle:
      \[\begin{tikzcd}
                {q_*\sheaf{\pro(\cF)}(-5H'+h')} & {q_*\sheaf{\pro(\cF)}(-3H'+2h')} & {\pi_*\sheaf{\tY}(-3H+2h)} & {}
                \arrow[from=1-1, to=1-2]
                \arrow[from=1-2, to=1-3]
                \arrow["{+1}", from=1-3, to=1-4]
            \end{tikzcd}\]
        Since $q_*\sheaf{\pro(\cF)}(-3H') = 0$, we have $\pi_*\sheaf{\tY}(-3H+2h)[1] \cong q_*\sheaf{\pro(\cF)}(-5H'+h')[2]$. The relative canonical bundle of the $\pro^4$-fibration $q\colon \pro(\cF) \to \pro^4$ is given by $\omega_q = q^*\det \cF^\dual \otimes \sheaf{\tY}(-5H) = \sheaf{\tY}(-5H+h)$. Hence by Grothendieck--Verdier duality,
        \[ \pi_*\sheaf{E}(-2H+h) \cong q_*\sheaf{\pro(\cF)}(-5H'+h')[2] \cong q_*\omega_q[2] = \sheaf{\pro^4}[-2].\]
        Therefore $\pi^*\pi_*\sheaf{E}(-2H+h) \otimes \sheaf{\tY}(3H) \cong \sheaf{\tY}(3H)[-2]$. Now by applying $\varPsi$ we have the distinguished triangle
        \[\begin{tikzcd}
                {\varPsi(\sheaf{\tY}(3H))[-2]} & \varPsi(\sheaf{E}(H+h)) & {\varPsi\circ\lmut{\beta_*\Dcatb(\pro^4)}\sheaf{E}(H+h)} & {}
                \arrow[from=1-1, to=1-2]
                \arrow[from=1-2, to=1-3]
                \arrow["{+1}", from=1-3, to=1-4]
            \end{tikzcd}\]
        For the second term, note that $\sheaf{E}(H+h)$ fits into the short exact sequence
        \[\begin{tikzcd}
                0 & {\sheaf{\tY}(2h)} & {\sheaf{\tY}(H+h)} & {\sheaf{E}(H+h)} & 0.
                \arrow[from=1-1, to=1-2]
                \arrow[from=1-2, to=1-3]
                \arrow[from=1-3, to=1-4]
                \arrow[from=1-4, to=1-5]
            \end{tikzcd}\]
        Applying $\varPsi$ to the sequence and using \cref{lem:psi_kill_7}, the first two terms vanish, and we deduce that $\varPsi(\sheaf{E}(H+h)) = 0$. Finally, using (1) we have:
        \[ \varPsi\circ\lmut{\beta_*\Dcatb(\pro^4)}\sheaf{E}(H+h) \cong \varPsi(\sheaf{\tY}(3H))[-1] \cong \CCC_2[2]. \]

        \item Consider the distinguished triangle:
        \begin{equation}
            \begin{tikzcd}[column sep = small]
             {\varPsi\lmut{\beta_*\Dcatb(\pro^4)}\sheaf{\tY}(3H+2h)} & {\varPsi\lmut{\beta_*\Dcatb(\pro^4)}\sheaf{\tY}(4H+h)} & {\varPsi\lmut{\beta_*\Dcatb(\pro^4)}\sheaf{E}(4H+h)} & {}
            \arrow[from=1-1, to=1-2]
            \arrow[from=1-2, to=1-3]
            \arrow["{+1}", from=1-3, to=1-4]
            \end{tikzcd}
        \end{equation}
        For the first term, we have that $\lmut{\beta_*\Dcatb(\pro^4)}\sheaf{\tY}(3H+2h) = 0$ since $\sheaf{\tY}(3H+2h) \in \beta_*\Dcatb(\pro^4)$. For the second term, the same computation as in (2) shows that $\varPsi\circ\lmut{\beta_*\Dcatb(\pro^4)}\sheaf{\tY}(4H+h) \cong \CCC_1(1)[4] \cong \CCC_3[4]$. Hence 
        \[
            \varPsi\circ\lmut{\beta_*\Dcatb(\pro^4)}\sheaf{E}(4H+h) \cong \varPsi\circ\lmut{\beta_*\Dcatb(\pro^4)}\sheaf{\tY}(4H+h) \cong \CCC_3[4]. \qedhere 
        \]
    \end{enumerate}
\end{proof}

\begin{proof}[Proof of \cref{thm:emb_Ku_7}.]
    By \cref{cor:Psi_on_mut_7}, $\Dcatb(\pro^4,\CCC_0)$ admits the semi-orthogonal decomposition:
    \begin{align*}
      \Dcatb(\pro^4,\CCC_0) &= \left\langle\varPsi\sigma^*\Ku(Y),\ 
      \varPsi\lmut{\beta_*\Dcatb(\pro^4)}\sheaf{}(4,0),\
      \varPsi\lmut{\beta_*\Dcatb(\pro^4)}\sheaf{E}(1,1),\ \varPsi\lmut{\beta_*\Dcatb(\pro^4)}\sheaf{E}(4,1)\right\rangle.
    \end{align*}
    By \cref{lem:C1_7} (2)--(4), this reduces to 
    \[ \Dcatb(\pro^4,\CCC_0) = \ord{ \varPsi\sigma^*\Ku(Y),\ \CCC_1,\ \CCC_2,\ \CCC_3}. \qedhere\]
\end{proof}

\section{Singular cubic 7-folds}\label{sec:sing_cubic_7}

Throughout this section, $Y$ is a singular cubic 7-fold as in \cref{def:sing_cubic}.

\begin{lemma}[][lem:sing_cubic_7]
  Let $Y$ be a cubic $7$-fold. $Y$ is singular along a line if and only if the defining equation of $Y$ can be projectively transformed into the form 
  \begin{equation}
    f(x_0,\ldots,x_8) = x_7f_1(x_0,\ldots,x_6) + x_8f_2(x_0,\ldots,x_6) + g(x_0,\ldots,x_6), \label{equ:sing_cub_7}
  \end{equation}
  where $f_1, f_2 \in \C[x_0,...,x_6]$ are homogeneous quadrics and $g \in \C[x_0,...,x_6]$ is a homogeneous cubic. Moreover, if $f_1, f_2, g$ are general, then $Y$ has transversally cuspidal singularities on $7$ points $p_1,...,p_7 \in \ell \coloneqq \mathbb{V}(x_0,\ldots,x_6)$, transversally nodal singularities on $\ell \setminus \{p_1,...,p_7\}$, and smooth elsewhere. 
\end{lemma}
\begin{proof}
  Suppose that $Y$ is singular along a line $\ell \subset Y$. By a projective transformation we may assume that $\ell = \mathbb{V}(x_0,\ldots,x_6)$. If $Y = \mathbb{V}(f)$, then the conditions that $f(0,\ldots,0,x_7,x_8) = \nabla f(0,\ldots,0,x_7,x_8) = 0$ implies that $f$ must take the form of \eqref{equ:sing_cub_7}. On the other hand, it is easy to verify that the equation \eqref{equ:sing_cub_7} defines a cubic 7-fold singular along $\ell$.

  Suppose that $f_1, f_2, g$ are general. We show that $\operatorname{Sing}(Y) = \ell$. If $Y$ is singular at a point $p \in Y \setminus \ell$, then the vanishing of $\pa_7 f = f_1$ and $\pa_8 f = f_2$ at $p$ implies that $p \in Z \coloneqq \mathbb{V}(f_1, f_2, g) \subset \pro^6$. Meanwhile, the vanishing of $\pa_i f$ for $i=0,...,6$ at $p$ implies that $\nabla f_1, \nabla f_2, \nabla g$ are linearly dependent at $p$. In particular, $Z$ is singular $p$. But by Bertini's theorem, $Z$ is smooth for general $f_1,f_2,g$, thus proving the claim.

  On the line $\ell$, the Hessian of $f$ is given by $\det(x_7H_1 + x_8H_2)$, where $H_i$ is the $7 \times 7$ matrix associated to the quadric $f_i$ for $i = 1,2$. For general $f_1,f_2$, the determinant $\det(x_7H_1 + x_8H_2)$ has exactly 7 distinct roots on the line $\ell$, which correspond to the 7 points $p_1,...,p_7$ where $Y$ has transversally cuspidal singularities. At these points the quadric $x_7f_1(x_0,\ldots,x_6) + x_8f_2(x_0,\ldots,x_6)$ has corank $1$ in $\pro^6$. At the other points on $\ell$, the Hessian does not vanish, hence $Y$ has transversally nodal singularities.
\end{proof}

\subsection{Geometric resolution}

Consider the blow-up $\sigma\colon \tY \to Y$ along the line $\ell \subset Y$. For the rational map $Y \dashto \pro^6$ associated to projection away from $\ell$, the blow-up $\sigma$ lifts it to a morphism $\pi\colon \tY \to \pro^6$. It is shown in the following diagram:

\begin{equation}
  \begin{tikzcd}[ampersand replacement=\&,cramped,column sep=small]
	\& E \&\& {\tilde{Y}} \&\& D \\
	\ell \&\& Y \&\& {\pro^{6}} \&\& Z
	\arrow["\iota", hook, from=1-2, to=1-4]
	\arrow["p"', from=1-2, to=2-1]
	\arrow["\sigma"', from=1-4, to=2-3]
	\arrow["\pi", from=1-4, to=2-5]
	\arrow["\eta"', hook', from=1-6, to=1-4]
	\arrow["s", from=1-6, to=2-7]
	\arrow["j", hook, from=2-1, to=2-3]
	\arrow[hook', from=2-7, to=2-5]
\end{tikzcd} \label{star_7}
\end{equation}

Let $H$ and $h$ be the divisor classes of $\tY$ corresponding to the pull-back of a hyperplane in $Y$ and $\pro^6$ respectively. 
From \cref{lem:sing_cubic_7} it is clear that the blow-up can be described by the following statements:
\begin{lemma}[][lem:blowup_7]
  The blow-up $\sigma\colon \tY \to Y$ along the line $\ell \subset Y$ is a resolution of singularities of $Y$. The exceptional divisor $E$ is a quadric fibration over $\ell$ with discriminant locus at the 7 points $p_1,...,p_7 \in \ell$. The canonical divisor of $\tY$ is given by 
  \[K_{\tY} = \sigma^*K_Y + (\codim(\ell \subset Y) - 2)E = -6H + 4E = -2H - 4h.\]
\end{lemma}
Note that $E = \left\{ ([x_0:\cdots:x_6], [x_7:x_8]) \mid x_7f_1(x_0,\ldots,x_6) + x_8f_2(x_0,\ldots,x_6) = 0 \right\} \subset \pro^6 \times \pro^1$. The quadric fibration $p\colon E \to \ell$ gives rise to a sheaf of Clifford algebras on $\ell$, with the even part $\sB_0$ given by 
\[\sB_0 \coloneqq \bigoplus_{m=0}^\infty \textstyle{\bigwedge^{2m}}\sheaf{\ell}^{\oplus 7} \otimes \sheaf{\ell}(-m) = \sheaf{\ell}
\oplus \sheaf{\ell}(-1)^{\oplus 21}
\oplus \sheaf{\ell}(-2)^{\oplus 35}
\oplus \sheaf{\ell}(-3)^{\oplus 7}.\]

Next we describe the fibration $\pi\colon \tY \to \pro^6$. Consider the embedded blow-up of $Y \subset \pro^8$ along $\ell$. We have that $\tY$ is a section of the projective bundle $\operatorname{Bl}_{\ell}\pro^8 \cong \pro_{\pro^6}(\sheaf{\pro^6}^{\oplus 2} \oplus \sheaf{\pro^6}(-1))$. On the other hand, let $\EEE \coloneqq \sheaf{\pro^6}(-2)^{\oplus 2} \oplus \sheaf{\pro^6}(-3)$. Twisting the projective bundle by $\sheaf{}(-2)$, we have that $\operatorname{Bl}_{\ell}\pro^8 \cong \pro(\EEE)$. The defining section $s \in \Hlg^0(\pro(\EEE), \sheaf{\pro(\EEE)}(1))$ of $\tY$ also defines a section $s' \in \Hlg^0(\pro^6,\EEE^\dual)$ via the identification $\Hlg^0(\pro(\EEE), \sheaf{\pro(\EEE)}(1)) \cong \Hlg^0(\pro^6,\EEE^\dual)$. The zero locus $Z$ of $s'$ is exactly given by the equations $f_1 = f_2 = g = 0$ in $\pro^6$. By adjunction formula, we have that $K_Z = -7h + 2h + 2h + 3h = 0$. In summary, we have the following description:

\begin{lemma}[][lem:fib_P6_7]
  $\tY$ can be identified with a section of the projective bundle $\pro_{\pro^6}(\sheaf{\pro^6}(-2)^{\oplus 2} \oplus \sheaf{\pro^6}(-3))$. In particular, the morphism $\pi\colon \tY \to \pro^6$ is a $\pro^1$-fibration over $\pro^6 \setminus Z$ and a $\pro^2$-fibration over $Z$, where $Z = \mathbb{V}(f_1,f_2,g) \subset \pro^6$ is a smooth Calabi--Yau 3-fold.
\end{lemma}

Let $D \coloneqq \pi^{-1}(Z)$. Note that $D$ has codimension $2$ in $\tY$.

\subsection{Resolution via the dual Lefschetz decomposition}

Now we describe a categorical resolution of the category $\Ku(Y)$ using the left-hand side of the diagram \eqref{star_7}.
\begin{definition}
    Let $X$ be a projective variety and $L$ a line bundle on $X$. A \textbf{dual Lefschetz decomposition} of $\Dcatb(X)$ is a semi-orthogonal decomposition of the form
    \[\Dcatb(X) = \ord{\BBB_{m-1} \otimes L^{1-m},...,\BBB_1 \otimes L^{-1},\BBB_0},\]
    where $\BBB_0,...,\BBB_m$ are full triangulated subcategories of $\Dcatb(X)$ satisfying:
    \[ \BBB_{m-1} \subseteq \cdots \subset \BBB_1 \subset \BBB_0. \]
\end{definition}

The quadric fibration $p\colon E \to \ell$ gives rise to a dual Lefschetz decomposition of $\Dcatb(E)$:

\begin{lemma}[][]
    Let $p\colon E \to \ell$ be a flat fibration in $5$-dimensional quadrics. Then $\Dcatb(E)$ admits a dual Lefschetz decomposition of the form
    \[\Dcatb(E) = \ord{\BBB_{4} \otimes \cL^{-4},\ \BBB_{3}\otimes \cL^{-3},\ \BBB_{2}\otimes \cL^{-2},\ \BBB_1\otimes \cL^{-1},\ \BBB_0},\]
    where $\BBB_{4} = \BBB_3 = \BBB_2 = \BBB_1 = \ord{p^*\Dcatb(\ell)}$ and $\BBB_0 = \ord{\varPhi \Dcatb(\ell, \sB_0),\ p^*\Dcatb(\ell)}$, and $\cL$ is the conormal bundle $\Nor{E}{\tY}^\dual = \sheaf{E}(-E) = \sheaf{E}(-H + h)$.
\end{lemma}
\begin{proof}
  By \cite{kuzQuaFib}, there is an SOD:
  \begin{equation}\label{equ:quad_fib_7}
    \Dcatb(E) = \ord{\varPhi\Dcatb(\ell,\sB_0),\ p^*\Dcatb(\ell),\ p^*\Dcatb(\ell) \otimes \sheaf{E}(h),\ \ldots,\ p^*\Dcatb(\ell) \otimes \sheaf{E}(4h)}. 
  \end{equation}
  By adjunction formula, the canonical divisor of $E$ is given by 
  \[K_E = \iota^*K_{\tY} + E = -H -5h.\]
  We mutate the category $\ord{p^*\Dcatb(\ell) \otimes \sheaf{E}(h),\ \ldots,\ p^*\Dcatb(\ell) \otimes \sheaf{E}(4h)}$ through its right orthogonal using the Serre functor. It follows that 
  \[\Dcatb(E) = \ord{p^*\Dcatb(\ell) \otimes \sheaf{E}(-4h),\ \ldots,\ p^*\Dcatb(\ell) \otimes \sheaf{E}(-h),\ \varPhi\Dcatb(\ell,\sB_0),\ p^*\Dcatb(\ell)}.\]
  This agrees with the desired form as $\cL^{-1}=\sheaf{E}(E)=\sheaf{E}(H-h)$, and the additional twist by $\sheaf{E}(H)$ is pulled back from $\ell$ and hence can be absorbed into $p^*\Dcatb(\ell)$.
\end{proof}

\cite{kuzCatRes} provides a result on constructing a categorical resolution from the dual Lefschetz decomposition of the exceptional divisor, adapting to our situation gives:

\begin{proposition}[{\cite[Proposition 4.1, Theorem 4.4, Proposition 4.5]{kuzCatRes}}][thm:crep_cat_res]
  Consider $\Dcatb(E)$ equipped with the above dual Lefschetz decomposition. Then the functor $\iota_*\colon \Dcatb(E) \to \Dcatb(\tilde Y)$ is fully faithful on each component $\BBB_k \otimes \cL^{-k}$ for $1 \leq k \leq 4$, and $\Dcatb(\tilde Y)$ admits the semi-orthogonal decomposition:
    \[\Dcatb(\tilde Y) = \ord{\iota_*(\BBB_{4} \otimes \cL^{-4}),\, \iota_*( \BBB_{3} \otimes \cL^{-3} ),\, \iota_*( \BBB_{2} \otimes \cL^{-2} ),\, \iota_*( \BBB_1 \otimes \cL^{-1} ),\, \tilde\DDD},\]
    where $\tilde\DDD$ is a full triangulated subcategory of $\Dcatb(\tilde Y)$ with objects
    \[\tilde\DDD = \left\{ F \in \Dcatb(\tilde Y) \mid \iota^*F \in \BBB_0 = \ord{\varPhi \Dcatb(\ell, \CCC_0),\ p^*\Dcatb(\ell)}\right\}.\]
    Furthermore, $\sigma_*\colon \tilde\DDD \to \Dcatb(Y)$ is a categorical resolution.
\end{proposition}

Moreover the following results show that the resolution gives a Verdier localization.

\begin{proposition}[][prop:sing_cubic7_ker_abstract]
  The functor $\sigma_*\colon \tilde\DDD \to \Dcatb(Y)$ is a Verdier localization, with $\ker \sigma_* = \ord{\iota_*\varPhi\Dcatb(\ell, \sB_0)}_{\mathrm{thick}}$.
\end{proposition}
\begin{proof}
  By \cite[Theorem 5.2]{KS24}, in order to show that $\sigma_*$ is a Verdier localization, we need to check that $p_*\colon \Dcatb(E) \to \Dcatb(\ell)$ is a Verdier localization and that $\Rsf^i\sigma_*\sheaf{\tY}(-mE) = 0$ for all $m>0$ and $i>0$.

  Firstly we check that $p_*\colon \Dcatb(E) \to \Dcatb(\ell)$ is a Verdier localization. By \cite[Proposition I.1.3]{GZ}, it suffices to check that the counit $p_*p^! \to \id_{\Dcatb(\ell)}$ is an isomorphism. By Grothendieck--Verdier duality and projection formula, we have 
  \[ p_*p^!F = p_*p^*(F \otimes \omega_{E/\ell}[5]) = F \otimes p_*\omega_{E/\ell}[5]. \]
  So it suffices to check $p_*\omega_{E/\ell}[5] \cong \sheaf{\ell}$ or equivalently $p_*\sheaf{E} = \sheaf{\ell}$. By cohomology and base change theorem, we need to check that $\Hlg^\bullet(Q,\sheaf{Q}) = \C[0]$ for every fibre $Q$ of $p\colon E \to \ell$. This holds as $Q$ is a quadric either smooth or of corank 1.

  The other condition that $\Rsf\sigma_*\sheaf{\tY}(-mE) = \ideal_{E}^m[0]$ for all $m>0$ holds for a similar reason. Consider the short exact sequence
  \[\begin{tikzcd}[ampersand replacement=\&,cramped]
	0 \& {\sheaf{\tY}(-(m+1)E)} \& {\sheaf{\tY}(-mE)} \& {\iota_*\sheaf{E}(-mE)} \& 0
	\arrow[from=1-1, to=1-2]
	\arrow[from=1-2, to=1-3]
	\arrow[from=1-3, to=1-4]
	\arrow[from=1-4, to=1-5]
  \end{tikzcd}\]
  We have 
  \[ \Rsf^i(\iota_*\sigma_*)\sheaf{E}(-mE) = j_*\Rsf^i p_*\sheaf{E}(-mE) = 0,  \]
  for $i>0$ and $m>0$ because $\Hlg^i(Q,\sheaf{Q}(m)) = 0$. Now the sequence gives a surjection $\Rsf^i\sigma_*\sheaf{\tY}(-(m+1)E) \twoheadrightarrow \Rsf^i\sigma_*\sheaf{\tY}(-mE)$ for $m>0$. But for $m \gg 0$ we have $\Rsf^i\sigma_*\sheaf{\tY}(-mE) = 0$ by Serre vanishing. Therefore $\Rsf^i\sigma_*\sheaf{\tY}(-mE) = 0$ for all $m>0$ and $i>0$. The degree-zero identity $\Rsf^0\sigma_*\sheaf{\tY}(-mE) = \ideal_{E}^m$ follows from the same sequence and the fact that $\Rsf^0\sigma_*\iota_*\sheaf{E}(-mE) \cong \ideal_{E}^{m}/\ideal_E^{m+1}$.

  Finally, \cite[Theorem 5.2]{KS24} gives the kernel $\ker \sigma_*|_{\tilde\DDD} = \ord{\iota_*(\ker p_*)}_{\mathrm{thick}} \cap \tilde\DDD = \ord{\iota_*\varPhi\Dcatb(\ell, \sB_0)}_{\mathrm{thick}}$.
\end{proof}

Next we pass on to the Kuznetsov component of $Y$. Recall that $Y$ admits an SOD \eqref{SOD:cubic_7}. Since $\sigma^*\colon \Perf(Y) \to \tilde \DDD$ is fully faithful, $\{\sheaf{\tY},...,\sheaf{\tY}(5H)\}$ remains an exceptional sequence. We define $\widetilde{\Ku(Y)}$ to be their right orthogonal in $\tilde \DDD$. That is,
\begin{equation}
  \widetilde{\Ku(Y)} \coloneqq \left\{ F \in \Dcatb(\tY) \mid \iota^*F \in \ord{\varPhi \Dcatb(\ell, \sB_0),\ p^*\Dcatb(\ell)},\ \forall\, k=0,...,5,\ \Ext^\bullet(\sheaf{\tY}(kH),F) = 0 \right\},
\end{equation}
and we have an SOD of $\Dcatb(\tY)$ given by 
  \begin{align}
    \Dcatb(\tY) &= \ord{\iota_*(p^*\Dcatb(\ell) \otimes \sheaf{E}(-4h)),\ ...,\ \iota_*(p^*\Dcatb(\ell) \otimes \sheaf{E}(-h)),\ \ \widetilde{\Ku(Y)},\ \ \sheaf{\tY},\ ...,\ \sheaf{\tY}(5H)} \\ 
    &= \ord{\widetilde{\Ku(Y)},\ \ \sheaf{\tY},\ ...,\ \sheaf{\tY}(5H),\ \ \iota_*p^*\Dcatb(\ell),\ ...,\ \iota_*p^*\Dcatb(\ell) \otimes \sheaf{\tY}(3h)}. \label{equ:Ku_Y_resolv_7}
  \end{align}

\begin{corollary}[][]
  The functor $\sigma_*\big|_{\widetilde{\Ku(Y)}}\colon \widetilde{\Ku(Y)} \to \Ku(Y)$ is a Verdier localization with  $\ker {\left(\sigma_*|_{\widetilde{\Ku(Y)}}\right)} = \ord{\iota_*\varPhi\Dcatb(\ell, \sB_0)}_{\mathrm{thick}}$, and defines a categorical resolution of $\Ku(Y)$.
\end{corollary}
\begin{proof}
  To check that $\sigma_*\big|_{\widetilde{\Ku(Y)}}$ is a categorical resolution, it remains to check that $\sigma_*\Ku(\tY) \subset \Ku(Y)$ and that $\sigma^*(\Ku(Y)^{\mathsf{perf}}) \subset \widetilde{\Ku(Y)}$.

    For $F \in \Ku({\tY})$ and $k \in \{0,...,5\}$,
    \[\RHom(\sheaf{Y}(k),\sigma_*F) = \RHom(\sigma^*\sheaf{Y}(k),F) = \RHom(\sheaf{\tY}(kH),F) = 0.\]
    Hence $\sigma_*(\Ku({\tY})) \subset \ord{\sheaf{Y},...,\sheaf{Y}(5)}^\perp = \Ku(Y)$.

    For $G \in \Ku(Y)^{\mathsf{perf}} = \Ku(Y) \cap \Perf(Y)$, and $k \in \{0,...,5\}$,
    \[\RHom(\sheaf{\tY}(kH),\sigma^*G) = \RHom(\sigma^*\sheaf{Y}(k),\sigma^*G) \cong \RHom(\sheaf{Y}(k),\sigma_*\sigma^*G) \cong \RHom(\sheaf{Y}(k),G) = 0,\]
    where we used that $\sigma_*\sigma^* \cong \id$ on $\Perf(Y)$. Hence $\sigma^*(\Ku(Y)^{\mathsf{perf}}) \subset \tilde{\DDD} \cap \ord{\sheaf{\tY},...,\sheaf{\tY}(5H)}^\perp = \widetilde{\Ku(Y)}$.

    Finally we show that $\ker {\left(\sigma_*|_{\widetilde{\Ku(Y)}}\right)} = \ker \sigma_* \cap \widetilde{\Ku(Y)} = \ord{\iota_*\varPhi\Dcatb(\ell, \sB_0)}_{\mathrm{thick}}$ by showing that $\iota_*\varPhi\Dcatb(\ell, \sB_0)$ is right orthogonal to $\ord{\sheaf{\tY},...,\sheaf{\tY}(5H)}$. Indeed, 
    \[\RHom(\sheaf{\tY}(kH), \iota_*\varPhi\Dcatb(\ell, \sB_0)) = \RHom(\iota^*\sheaf{\tY}(kH), \varPhi\Dcatb(\ell, \sB_0)) = 0\]
    since $\iota^*\sheaf{\tY}(kH) \in p^*\Dcatb(\ell) \subset {}^{\perp}\varPhi\Dcatb(\ell, \sB_0)$ by \eqref{equ:quad_fib_7}. 
  \end{proof}

\subsection{Resolution via a Calabi--Yau 3-fold}

We will show that $\widetilde{\Ku(Y)}$ is equivalent to $\Dcatb(Z)$ by considering another semi-orthogonal decomposition of the derived category $\Dcatb(\tY)$. As a corollary, we show that the categorical resolution $\sigma_*\big|_{\widetilde{\Ku(Y)}}$ is weakly crepant.

Note that $\Pic \tY = \Z H \oplus \Z h$. Similar to the previous section, we denote the sheaves $\sheaf{\tY}(aH+bh)$ and $\iota_*\sheaf{E}(aH+bh)$ by $\sheaf{}(a,b)$ and $\sheaf{E}(a,b)$ respectively, where $a,b \in \Z$. For $a=b=0$, these sheaves will be denoted simply by $\sheaf{}$ and $\sheaf{E}$.

From the left-hand side of the diagram \eqref{star_7}, we obtain the SOD \eqref{equ:Ku_Y_resolv_7}. Using the full exceptional collections on $\Dcatb(\ell)$, \eqref{equ:Ku_Y_resolv_7} can be refined as follows:
\begin{equation}
  \begin{aligned}
    \Dcatb(\tY)
    = \langle\,
    &\widetilde{\Ku(Y)} \otimes \sheaf{\tY}(2H),\ 
      \sheaf{}(2,0),\ 
      \sheaf{}(3,0),\ 
      \sheaf{}(4,0),\ 
      \sheaf{}(5,0),\ 
      \sheaf{}(6,0),\ 
      \sheaf{}(7,0),\\
    &\sheaf{E}(2,0),\ 
      \sheaf{E}(3,0),\ 
      \sheaf{E}(2,1),\ 
      \sheaf{E}(3,1),\ 
      \sheaf{E}(3,2),\ 
      \sheaf{E}(4,2),\ 
      \sheaf{E}(3,3),\ 
      \sheaf{E}(4,3)
    \,\rangle.
\end{aligned} \label{equ:SOD_torsion_7}
\end{equation}
Note that we have applied an extra twist by $\sheaf{\tY}(2H)$ on $\tilde\DDD$ for later mutation use.

Then consider the right-hand side of the diagram \eqref{star_7}. By \cref{lem:fib_P6_7} and Orlov's Cayley trick (\cite[Proposition 2.10]{orlov06}), $\eta_*s^*\colon \Dcatb(Z) \to \Dcatb(\tY)$ is fully faithful, and $\Dcatb(\tY)$ admits an SOD:
\begin{equation}
  \Dcatb(\tY) = \ord{\eta_*s^*\Dcatb(Z),\ \pi^*\Dcatb(\pro^6) \otimes \sheaf{\tY}(\tY),\ \pi^*\Dcatb(\pro^6) \otimes \sheaf{\tY}(2\tY)}, \label{equ:proj_bun_7}
\end{equation}
where $\sheaf{\tY}(\tY) = \sheaf{\pro(\EEE)}(\tY)\big|_{\tY} = \sheaf{\tY}(H+2h) = \sheaf{}(1,2)$. Using the full exceptional collection $\langle\sheaf{}(-4),\ldots,\sheaf{}(2)\rangle$ for $\Dcatb(\pro^6)$ and pull it back via $\pi^*$, we obtain the SOD:
\begin{equation}
  \begin{aligned}
    \Dcatb(\tY) = \langle &\eta_*s^*\Dcatb(Z),\ \sheaf{}(1,-4),\ \sheaf{}(1,-3),\ \sheaf{}(1,-2),\ \sheaf{}(1,-1),\ \sheaf{}(1,0),\ \sheaf{}(1,1),\ \sheaf{}(1,2),\\
    &\sheaf{}(2,-2),\ \sheaf{}(2,-1),\ \sheaf{}(2,0),\ \sheaf{}(2,1),\ \sheaf{}(2,2),\ \sheaf{}(2,3),\ \sheaf{}(2,4)\rangle.
  \end{aligned} \label{equ:SOD_fib_7}
\end{equation} 

\begin{proposition}[][prop:mut_sing_7]
  There is an equivalence $\varTheta\colon\Dcatb(Z) \to\widetilde{\Ku(Y)}$ given by 
  \[ \varTheta = (\sheaf{}(-2,0) \otimes -)\circ\rmut{\cC}\circ\eta_*\circ s^*,\]
  where $\cC \coloneqq \langle(1,-4),$ $%
      (1,-3),$ $%
      (1,-2),$ $%
      (1,-1),$ $%
      (1,0),$ $%
      (2,-2),$ $%
      (2,-1)\rangle$.
\end{proposition}

\begin{corollary}[][]
  The categorical resolution $\sigma_*\big|_{\widetilde{\Ku(Y)}}\colon \widetilde{\Ku(Y)} \to \Ku(Y)$ is weakly crepant.
\end{corollary}
\begin{proof}
  It suffices to check that $\ker{\left( \sigma_*\big|_{\widetilde{\Ku(Y)}} \right)} \subset \widetilde{\Ku(Y)}$ is Serre-invariant (\cite[Lemma 5.8]{KS24_hf}),  which is trivially true as $\widetilde{\Ku(Y)} \simeq \Dcatb(Z)$ is a Calabi--Yau category.
\end{proof}

Before going into the proof of \cref{prop:mut_sing_7}, we prove and summarize some orthogonality and mutation relations of line bundles and torsion sheaves on $\Dcatb(\tY)$.

\begin{lemma}[][lem:ortho_OY_sing_7]
  The orthogonality relation $\Ext^\bullet(\sheaf{}(a,b),\sheaf{}) = 0$ holds on $\tY$ if:
  \begin{enumerate}[nosep, before=\vspace*{-\parskip}]
    \item either $a = 1$;
    \item or $(a,b) \in \{ (-3,4),$ $(-2,3),$ $(-2,4),$ $(-1,2),$ $(-1,3),$ $(-1,4),$ $(0,1),$ $(0,2),$ $(0,3),$ $(0,4),$ $(0,5),$ $(0,6),$ $(2,-2), (2,-1), (2,0),$ $(2,1),$ $(2,2),$ $(2,3),$ $(3,0),$ $(3,1),$ $(3,2),$ $(4,0),$ $(4,1),$ $(5,0) \}$.
  \end{enumerate}
  In particular, {$\sheaf{}$ is completely orthogonal to $\sheaf{}(1,-2), \sheaf{}(1,-3), \sheaf{}(1,-4)$.}
\end{lemma}
\begin{proof}
  Firstly we note that $\Ext^\bullet(\sheaf{}(1,b),\sheaf{}) = 0$ for any $b \in \Z$ by the projective bundle formula \eqref{equ:proj_bun_7}. 
  
  Secondly we note that $\Ext^\bullet(\sheaf{}(0,b),\sheaf{}) = \Hlg^\bullet(\pro^6,\sheaf{\pro^6}(-b)) = 0$ for $1 \leq b \leq 6$.
  
  Thirdly, we claim that $\Ext^\bullet(\sheaf{}(a,b),\sheaf{}) = 0$ for $(a,b) \in \{ (-3,4),$ $(-2,3),$ $(-2,4),$ $(-1,2),$ $(-1,3),$ $(-1,4)\}$. Indeed, we have that $\tY = H + 2h$ as a divisor in the embedded blow-up space $\operatorname{Bl_{\ell}\pro^8} = \pro(\EEE)$. Consider the short exact sequence 
  \[\begin{tikzcd}[ampersand replacement=\&,cramped]
	0 \& {\sheaf{\pro(\EEE)}(-(a+1)H-(b+2)h)} \& {\sheaf{\pro(\EEE)}(-aH-bh)} \& {\sheaf{\tilde{Y}}(-aH-bh)} \& 0
	\arrow[from=1-1, to=1-2]
	\arrow[from=1-2, to=1-3]
	\arrow[from=1-3, to=1-4]
	\arrow[from=1-4, to=1-5]
  \end{tikzcd}\]
  Let us compute the sheaf cohomology on $\pro(\EEE)$. For $a \leq 0$, we have 
  \begin{align*}
    \Hlg^\bullet(\pro(\EEE),\sheaf{\pro(\EEE)}(-aH-bh)) 
    &= \Hlg^\bullet(\pro(\EEE),\sheaf{\pro(\EEE)}(-a\tY - (b-2a)h)) \\
    &= \Hlg^\bullet\left(\pro^6, \Sym^{-a}\left( \sheaf{\pro^6}(2)^{\oplus 2} \oplus \sheaf{\pro^6}(3) \right) \otimes \sheaf{\pro^6}(2a-b)\right) \\ 
    &= \bigoplus_{j=0}^{-a} \Hlg^\bullet\left( \pro^6,  \sheaf{\pro^6}(-b+j)^{\oplus(-a-j+1)} \right)
  \end{align*}
  The direct summands vanish for $0 \leq j \leq -a$ whenever $-6 \leq -b+j \leq -1$. In particular if $1-a \leq b \leq 6$, then $\Hlg^\bullet(\pro(\EEE),\sheaf{\pro(\EEE)}(-aH-bh)) = 0$. Substituting this condition into the above sequence, we observe that if $1-a \leq b \leq 4$, then $\Hlg^\bullet(\tY, \sheaf{\tilde{Y}}(-aH-bh)) = 0$. This condition is satisfied for $(a,b)$ in the claimed range.

  Finally, since $K_{\tY} = -2H-4h$, by Serre duality on $\tY$ we have that
  \[\Ext^\bullet(\sheaf{}(a,b),\sheaf{}) = \Ext^\bullet(\sheaf{}(2-a,4-b),\sheaf{}[7])^\dual.\]
  Hence $\Ext^\bullet(\sheaf{}(a,b),\sheaf{}) = 0$ also holds for $(a,b) \in \{ (2,-2)$, $(2,-1), (2,0),$ $(2,1),$ $(2,2),$ $(2,3),$ $(3,0),$ $(3,1),$ $(3,2),$ $(4,0),$ $(4,1),$ $(5,0) \}$.
\end{proof}

\begin{figure}[h]
    \centering
    \begin{minipage}{0.48\textwidth}
        \centering
        \begin{tikzpicture}[
  x=0.65cm,y=0.65cm,
  every node/.style={font=\scriptsize},
  ortho point/.style={circle, draw=green!55!black, fill=white, line width=1.2pt, inner sep=1.5pt},
  ortho line/.style={green!55!black, line width=1.2pt},
  axis/.style={black, line width=.55pt, ->},
  grid line/.style={gray!35, line width=.25pt}
]
  \def\xmin{-3.6}\def\xmax{5.6}\def\ymin{-2.4}\def\ymax{6.75}

  \draw[step=1, grid line] (\xmin,\ymin) grid (\xmax,\ymax);

  \draw[ortho line] (1,\ymin) -- (1,\ymax);

  \draw[axis] (\xmin,0) -- (\xmax+0.22,0) node[right] {$H$};
  \draw[axis] (0,\ymin) -- (0,\ymax+0.22) node[above] {$h$};

  \foreach \x/\xs in {-3/0,-2/0,-1/0,0/-5,1/-5,2/0,3/0,4/0,5/0}{
    \draw (\x,.06)--(\x,-.06) node[below=2pt, xshift=\xs pt] {$\x$};
  }
  \foreach \y in {-2,-1,1,2,3,4,5,6}{
    \draw (.06,\y)--(-.06,\y) node[left=2pt] {$\y$};
  }

  \foreach \y in {-2,-1,0,1,2,3,4,5,6}{
    \node[ortho point] at (1,\y) {};
  }

  \foreach \x/\y in {
    -3/4,
    -2/3,-2/4,
    -1/2,-1/3,-1/4,
     0/1,0/2,0/3,0/4,0/5,0/6,
     2/-2,2/-1,2/0,2/1,2/2,2/3,
     3/0,3/1,3/2,
     4/0,4/1,
     5/0
  }{
    \node[ortho point] at (\x,\y) {};
  }

  \node[green!55!black, above = 1pt, fill=white, inner sep=1pt]
    at (1,\ymax) {$a=1$};
\end{tikzpicture}
        \caption[Singular cubic 7-fold: line bundles left orthogonal to $\sheaf{}$.]{Line bundles left orthogonal to $\sheaf{}$.}
        \label{fig:OY2_7}
    \end{minipage}%
    \begin{minipage}{0.48\textwidth}
        \centering
        \begin{tikzpicture}[
  x=0.65cm,y=0.65cm,
  every node/.style={font=\scriptsize},
  ortho point/.style={circle, draw=green!55!black, fill=white, line width=1.2pt, inner sep=1.5pt},
  ortho line/.style={green!55!black, line width=1.2pt},
  axis/.style={black, line width=.55pt, ->},
  grid line/.style={gray!35, line width=.25pt}
]
  \def\xmin{-3.6}\def\xmax{5.6}\def\ymin{-2.4}\def\ymax{6.75}

  \draw[step=1, grid line] (\xmin,\ymin) grid (\xmax,\ymax);

  \foreach \y in {1,2,3,4}{
    \draw[ortho line] (\xmin,\y) -- (\xmax,\y);
  }

  \draw[axis] (\xmin,0) -- (\xmax+0.22,0) node[right] {$H$};
  \draw[axis] (0,\ymin) -- (0,\ymax+0.22) node[above] {$h$};

  \foreach \x/\xs in {-3/0,-2/0,-1/0,0/-5,1/0,2/0,3/0,4/0,5/0}{
    \draw (\x,.06)--(\x,-.06) node[below=2pt, xshift=\xs pt] {$\x$};
  }
  \foreach \y/\ys in {-2/0,-1/0,1/-5,2/-5,3/-5,4/-5,5/0,6/0}{
    \draw (.06,\y)--(-.06,\y) node[left=2pt, yshift=\ys pt] {$\y$};
  }

  \foreach \y in {1,2,3,4}{
    \foreach \x in {-3,-2,-1,0,1,2,3,4,5}{
      \node[ortho point] at (\x,\y) {};
    }
  }

  \foreach \x/\y in {
    0/5,0/6,
    1/-1,1/0
  }{
    \node[ortho point] at (\x,\y) {};
  }

  \node[green!55!black, anchor=west, inner sep=1pt]
    at (3.4,4.5) {$b=1,2,3,4$};
\end{tikzpicture}
        \caption[Singular cubic 7-fold: line bundles left orthogonal to $\sheaf{E}$.]{Line bundles left orthogonal to $\sheaf{E}$.}
        \label{fig:OE2_7}
    \end{minipage}
\end{figure}

\begin{lemma}[][lem:ortho_OE_sing_7]
  The orthogonality relation $\Ext^\bullet(\sheaf{}(a,b),\sheaf{E}) = 0$ holds on $\tY$ if:
  \begin{enumerate}[nosep, before=\vspace*{-\parskip}]
    \item either $b \in \{1,2,3,4\}$;
    \item or $(a,b) \in \{ (0,5),\ (0,6),\ (1,-1),\ (1,0) \}$.
  \end{enumerate}
  In particular, $\sheaf{E}$ is completely orthogonal to $\sheaf{}(-2,1), \sheaf{}(-2,2), \sheaf{}(1,-1), \sheaf{}(1,0)$.
\end{lemma}
\begin{proof}
  If $b \in \{1,2,3,4\}$, then $\Ext^\bullet(\sheaf{}(a,b),\sheaf{E}) = 0$ by the SOD \eqref{equ:quad_fib_7}. For the remaining cases, consider the short exact sequence: 
  \[\begin{tikzcd}[ampersand replacement=\&,cramped]
	0 \& {\sheaf{}(-1,1)} \& {\sheaf{}} \& {\sheaf{E}} \& 0
	\arrow[from=1-1, to=1-2]
	\arrow[from=1-2, to=1-3]
	\arrow[from=1-3, to=1-4]
	\arrow[from=1-4, to=1-5]
\end{tikzcd}\]
  When $(a,b) \in \{ (0,5),\ (0,6),\ (1,-1),\ (1,0) \}$, by the previous lemma we have that $\Ext^\bullet(\sheaf{}(a+1,b-1),\sheaf{}) = \Ext^\bullet(\sheaf{}(a,b),\sheaf{}) = 0$. Therefore $\Ext^\bullet(\sheaf{}(a,b),\sheaf{E}) = 0$ by the sequence.

  Finally, the complete orthogonality follows from Serre duality:
  \[ \Ext^\bullet(\sheaf{E},\sheaf{}(a,b)) = \Ext^\bullet(\sheaf{}(a+2,b+4),\sheaf{E}[7])^\dual. \qedhere \]
\end{proof}

Next we state the required mutation relations:
\begin{lemma}[][lem:mut_sing_7]
  \begin{enumerate}[nosep, before=\vspace*{-\parskip}]
    \item $\lmut{\sheaf{}}\sheaf{E} = \sheaf{}(-1,1)[1]$;
    \item $\rmut{\iota_*p^*\Dcatb(\ell)}\sheaf{}(a,0) = \sheaf{}(a-1,1)$ for any $a \in \Z$.
  \end{enumerate}
\end{lemma}
\begin{proof}
  For (1) it follows from \cref{lem:mut_trig} and the short exact sequence 
  \[\begin{tikzcd}[ampersand replacement=\&,cramped]
	0 \& {\sheaf{}(-1,1)} \& {\sheaf{}} \& {\sheaf{E}} \& 0.
	\arrow[from=1-1, to=1-2]
	\arrow[from=1-2, to=1-3]
	\arrow[from=1-3, to=1-4]
	\arrow[from=1-4, to=1-5]
 \end{tikzcd} \]
  For (2), we take a decomposition $\iota_*p^*\Dcatb(\ell) = \ord{\sheaf{E}(a-1,0),\sheaf{E}(a,0)}$. Then we have 
  \[ \rmut{\iota_*p^*\Dcatb(\ell)}\sheaf{}(a,0) = \rmut{\sheaf{E}(a,0)}\rmut{\sheaf{E}(a-1,0)}\sheaf{}(a,0) = \rmut{\sheaf{E}(a,0)}\sheaf{}(a,0) = \sheaf{}(a-1,1),\]
  where the second equality follows from that $\Ext^\bullet(\sheaf{}(a,0),\sheaf{E}(a-1,0)) = 0$ and third equality also from the sequence above.
\end{proof}

\begin{proof}[{Proof of \cref{prop:mut_sing_7}}]
We use the shorthand
\[
  (a,b) \coloneqq \sheaf{}(a,b)=\sheaf{\tY}(aH+bh),\qquad [a,b] \coloneqq \sheaf{E}(a,b) = \sheaf{E}(aH+bh),
\]
for all integers $a,b$. We prove the proposition by an explicit sequence of mutations. The goal is to match the $14$ exceptional objects of the SOD \eqref{equ:SOD_torsion_7} and those of the SOD \eqref{equ:SOD_fib_7}. We will start with \eqref{equ:SOD_torsion_7}.

\begin{enumerate}[label=\textbf{Step \arabic*.}, leftmargin=*]
    \item Using \cref{lem:mut_sing_7}.(2), we right mutate:
    \begin{enumerate}[nosep, before=\vspace*{-\parskip}]
      \item $(7,0)$ through $\ord{\iota_*p^*\Dcatb(\ell),\ \iota_*p^*\Dcatb(\ell) \otimes (1,1),\ \iota_*p^*\Dcatb(\ell) \otimes (1,2)}$;
      \item $\ord{(5,0),\ (6,0)}$ through $\ord{\iota_*p^*\Dcatb(\ell),\ \iota_*p^*\Dcatb(\ell) \otimes (1,1)}$; 
      \item $(4,0)$ through $\iota_*p^*\Dcatb(\ell)$. 
    \end{enumerate}
    We obtain the SOD:
    \begin{equation}
      \begin{tikzcd}[ampersand replacement=\&,cramped,sep=0]
    	{\Dcatb(\tY) = \langle        \widetilde{\Ku(Y)} \otimes (2,0),} \& {(2,0),} \& {(3,0),} \& {[2,0],} \& {[3,0],} \& {(3,1),} \& {[2,1],} \& {[3,1],} \\
    	\& {(3,2),} \& {(4,2),} \& {[3,2],} \& {[4,2],} \& {(4,3),} \& {[3,3],} \& {[4,3]\rangle.}
    \end{tikzcd}
    \end{equation}

    \item Using \cref{lem:ortho_OE_sing_7}, we transpose the pairs $\ord{(3,0),[2,0]}$, $\ord{(2,1),[3,1]}$, $\ord{(4,2),[3,2]}$, and $\ord{(4,3),[3,3]}$:
    \begin{equation}
      \begin{tikzcd}[ampersand replacement=\&,cramped,sep=0]
    	{\Dcatb(\tY) = \langle        \widetilde{\Ku(Y)} \otimes (2,0),} \& {(2,0),} \& {[2,0],} \& {(3,0),} \& {[3,0],} \& {[2,1],} \& {(3,1),} \& {[3,1],} \\
    	\& {(3,2),} \& {[3,2],} \& {(4,2),} \& {[4,2],} \& {[3,3],} \& {(4,3),} \& {[4,3]\rangle.}
    \end{tikzcd}
    \end{equation}

    \item Using \cref{lem:mut_sing_7}.(1), we left mutate: (i) $[2,0]$ through $(2,0)$; (ii) $[3,0]$ through $(3,0)$; (iii) $[3,1]$ through $(3,1)$; (iv) $[3,2]$ through $(3,2)$; (v) $[4,2]$ through $(4,2)$; and (vi) $[4,3]$ through $(4,3)$.
    \begin{equation}
      \begin{tikzcd}[ampersand replacement=\&,cramped,sep=0]
    	{\Dcatb(\tY) = \langle        \widetilde{\Ku(Y)} \otimes (2,0),} \& {(1,1),} \& {(2,0),} \& {(2,1),} \& {(3,0),} \& {[2,1],} \& {(2,2),} \& {(3,1),} \\
    	\& {(2,3),} \& {(3,2),} \& {(3,3),} \& {(4,2),} \& {[3,3],} \& {(3,4),} \& {(4,3)\rangle.}
    \end{tikzcd}
    \end{equation}

    \item Using \cref{lem:ortho_OE_sing_7}, we transpose the pairs $\ord{(3,0),[2,1]}$ and $\ord{(4,2),[3,3]}$:
    \begin{equation}
    \begin{tikzcd}[ampersand replacement=\&,cramped,sep=0]
    	{\Dcatb(\tY) = \langle        \widetilde{\Ku(Y)} \otimes (2,0),} \& {(1,1),} \& {(2,0),} \& {(2,1),}\& {[2,1],}  \& {(3,0),} \& {(2,2),} \& {(3,1),} \\
    	\& {(2,3),} \& {(3,2),} \& {(3,3),} \& {[3,3],} \& {(4,2),} \& {(3,4),} \& {(4,3)\rangle.}
    \end{tikzcd}
    \end{equation}

    \item Using \cref{lem:mut_sing_7}.(1), we left mutate: (i) $[2,1]$ through $(2,1)$; and (ii)  $[3,3]$ through $(3,3)$.
    \begin{equation}
    \begin{tikzcd}[ampersand replacement=\&,cramped,sep=0]
    	{\Dcatb(\tY) = \langle        \widetilde{\Ku(Y)} \otimes (2,0),} \& {(1,1),} \& {(2,0),} \& {(1,2),}   \& {(2,1),} \& {(3,0),} \& {(2,2),} \& {(3,1),} \\
    	\& {(2,3),} \& {(3,2),}\& {(2,4),}  \& {(3,3),} \& {(4,2),} \& {(3,4),} \& {(4,3)\rangle.}
    \end{tikzcd}
    \end{equation}

    \item Using \cref{lem:ortho_OY_sing_7}, we permute the line bundles using completely orthogonal pairs:
    \begin{equation}
    \begin{tikzcd}[ampersand replacement=\&,cramped,sep=0]
    	{\Dcatb(\tY) = \langle        \widetilde{\Ku(Y)} \otimes (2,0),} \& {(1,1),}\& {(1,2),}  \& {(2,0),}  \& {(2,1),} \& {(2,2),} \& {(2,3),} \& {(2,4),} \\
    	\& {(3,0),} \& {(3,1),} \& {(3,2),} \& {(3,3),} \& {(3,4),} \& {(4,2),} \& {(4,3)\rangle.}
    \end{tikzcd}
    \end{equation}

    \item Left mutate $\ord{(3,0),\ 
      (3,1),\ 
      (3,2),\ 
      (3,3),\ 
      (3,4),\ 
      (4,2),\ 
      (4,3)}$ through its right orthogonal, using the Serre functor $\sfS = (-2,-4)[7]$. We obtain:
    \begin{equation}
    \begin{aligned}
    \Dcatb(\tY)
    = \langle\,
    (1,-4),\ 
      (1,-3),\ 
      (1,-2),\ 
      (1,-1),\ 
      (1,0),\ 
      (2,-2),\ 
      (2,-1),\
      \\
        \widetilde{\Ku(Y)} \otimes (2,0),\ (1,1),\ 
      (1,2),\
      (2,0),\ 
      (2,1),\ 
      (2,2),\ 
      (2,3),\ 
      &(2,4)
    \,\rangle.
    \end{aligned}
    \end{equation}

    \item Left mutate $\widetilde{\Ku(Y)} \otimes (2,0)$ through $\cC \coloneqq \langle(1,-4),$ $%
      (1,-3),$ $%
      (1,-2),$ $%
      (1,-1),$ $%
      (1,0),$ $%
      (2,-2),$ $%
      (2,-1)\rangle$:
    \begin{equation}
    \begin{aligned}
    \Dcatb(\tY)
    = \langle\,
        \lmut{\cC}(\widetilde{\Ku(Y)} \otimes (2,0)),\ 
      (1,-4),\ 
      (1,-3),\ 
      (1,-2),\ 
      (1,-1),\ 
      (1,0),\ 
      (2,-2),\ 
      &(2,-1),\\ 
      (1,1),\ 
      (1,2),\
      (2,0),\ 
      (2,1),\ 
      (2,2),\ 
      (2,3),\ 
      &(2,4)
    \,\rangle.
    \end{aligned}
    \end{equation}

    \item Using \cref{lem:ortho_OY_sing_7}, transpose the completely orthogonal pair $\ord{(1,1),(1,2)}$ and $\ord{(2,-2),(2,-1)}$.
    \begin{equation}
    \begin{aligned}
    \Dcatb(\tY)
    = \langle\,
        \lmut{\cC}(\widetilde{\Ku(Y)} \otimes (2,0)),\ 
      (1,-4),\ 
      (1,-3),\ 
      (1,-2),\ 
      (1,-1),\ 
      (1,0),\ 
      (1,1),\ 
      &(1,2),\\
      (2,-2),\ 
      (2,-1),\ 
      (2,0),\ 
      (2,1),\ 
      (2,2),\ 
      (2,3),\ 
      &(2,4)
    \,\rangle.
    \end{aligned}
    \end{equation}

    Comparing the above SOD with \eqref{equ:SOD_fib_7}, we conclude that
    \[ \lmut{\cC}(\widetilde{\Ku(Y)} \otimes (2,0)) = \eta_*s^*\Dcatb(Z).\]
    Hence $\varTheta \coloneqq ((-2,0) \otimes -)\circ\rmut{\cC}\circ\eta_*\circ s^*$ gives the equivalence $\Dcatb(Z) \xrightarrow{\sim} \widetilde{\Ku(Y)}$. \qedhere
\end{enumerate}
\end{proof}

\begin{figure}[h]
    \centering
    \resizebox{0.9\textwidth}{!}{%
        \begin{tikzpicture}[
  x=0.50cm,y=0.50cm,
  >=Stealth,
  every node/.style={font=\scriptsize},
  linept/.style={circle, draw=green!55!black, fill=white, line width=.75pt, inner sep=1.35pt},
  ghostpt/.style={circle, draw=gray!60, fill=white, line width=.55pt, inner sep=1.25pt},
  torspt/.style={text=red!75!black, font=\scriptsize\bfseries, inner sep=0pt},
  bothpt/.style={text=blue!45!red!85!black, font=\scriptsize\bfseries, inner sep=0pt},
  axis/.style={black, line width=.55pt, -{Stealth[length=1.7mm,width=1.2mm]}},
  gridline/.style={gray!28, line width=.23pt},
  rmut/.style={orange!85!black, dashed, line width=.72pt,
    shorten <=2.4pt, shorten >=2.4pt,
    -{Stealth[length=1.75mm,width=1.25mm]}},
  lmut/.style={orange!85!black, line width=.78pt,
    shorten <=2.5pt, shorten >=2.5pt,
    -{Stealth[length=1.75mm,width=1.25mm]}},
  orth/.style={green!55!black, densely dashed, line width=.62pt, opacity=.76},
  serre/.style={teal!65!black, densely dashed, line width=.75pt,
    shorten <=2.5pt, shorten >=2.5pt,
    -{Stealth[length=1.75mm,width=1.25mm]}},
  orderguide/.style={green!55!black, line width=.58pt, opacity=.46,
    shorten <=1.0pt, shorten >=1.0pt,
    -{Stealth[length=1.8mm,width=1.15mm]}},
  paneltitle/.style={font=\small\bfseries, anchor=south, fill=white, inner sep=1pt},
  tick/.style={font=\tiny, fill=white, inner sep=.45pt},
  orthlabel/.style={font=\scriptsize, text=green!45!black, fill=white, inner sep=.6pt}
]

\def\Axes#1#2#3#4{%
  \draw[step=1, gridline] (#1,#3) grid (#2,#4);
  \draw[axis] (#1,0) -- ({#2+.35},0) node[right] {$H$};
  \draw[axis] (0,#3) -- (0,{#4+.35}) node[above] {$h$};
}
\def\Xticks#1{%
  \foreach \x in {#1}{
    \draw (\x,.06)--(\x,-.06) node[tick, below=2pt] {$\x$};
  }
}
\def\Yticks#1{%
  \foreach \y in {#1}{
    \draw (.06,\y)--(-.06,\y) node[tick, left=2pt] {$\y$};
  }
}
\def\Odot#1#2{\node[linept] at (#1,#2) {};}
\def\Gdot#1#2{\node[ghostpt] at (#1,#2) {};}
\def\Edot#1#2{\node[torspt] at (#1,#2) {$\times$};}
\def\Both#1#2{\node[bothpt] at (#1,#2) {$\star$};}
\def\PlotO#1{\foreach \x/\y in {#1}{\Odot{\x}{\y}}}
\def\PlotG#1{\foreach \x/\y in {#1}{\Gdot{\x}{\y}}}
\def\PlotE#1{\foreach \x/\y in {#1}{\Edot{\x}{\y}}}
\def\PlotB#1{\foreach \x/\y in {#1}{\Both{\x}{\y}}}

\begin{scope}[xshift=0cm,yshift=0cm]
  \Axes{-.35}{7.35}{-.55}{3.35}
  \Xticks{0,1,2,3,4,5,6,7}
  \Yticks{1,2,3}
  \node[paneltitle] at (2,3.92) {Step $0\to1$};

  \draw[rmut] (4,0) -- (3,1);
  \draw[rmut] (5,0) -- (4,1) -- (3,2);
  \draw[rmut] (6,0) -- (5,1) -- (4,2);
  \draw[rmut] (7,0) -- (6,1) -- (5,2) -- (4,3);

  \PlotO{4/0,5/0,6/0,7/0}
  \PlotE{2/1,3/3,3/1,3/2,4/2,4/3}
  \PlotB{2/0,3/0}
\end{scope}

\begin{scope}[xshift=5cm,yshift=0cm]
  \Axes{-.35}{4.35}{-.55}{3.35}
  \Xticks{0,1,2,3,4}
  \Yticks{1,2,3}
  \node[paneltitle] at (1.5,3.92) {Step $2$};

  \PlotE{2/1,3/3}
  \PlotB{2/0,3/0,3/1,3/2,4/2,4/3}
\end{scope}

\begin{scope}[xshift=9cm,yshift=0cm]
  \Axes{-.35}{4.35}{-.55}{4.35}
  \Xticks{0,1,2,3,4}
  \Yticks{1,2,3,4}
  \node[paneltitle] at (1.5,4.92) {Step $3$};

  \draw[lmut] (2,0) -- (1,1);
  \draw[lmut] (3,0) -- (2,1);
  \draw[lmut] (3,1) -- (2,2);
  \draw[lmut] (3,2) -- (2,3);
  \draw[lmut] (4,2) -- (3,3);
  \draw[lmut] (4,3) -- (3,4);

  \PlotO{2/0,1/1,3/0,2/2,3/1,2/3,3/2,4/2,3/4,4/3}
  \PlotB{2/1,3/3}
\end{scope}

\begin{scope}[xshift=0cm,yshift=-4cm]
  \Axes{-.35}{4.35}{-.55}{4.35}
  \Xticks{0,1,2,3,4}
  \Yticks{1,2,3,4}
  \node[paneltitle] at (2,4.92) {Step $4\to5$};

  \draw[lmut] (2,1) -- (1,2);
  \draw[lmut] (3,3) -- (2,4);

  \PlotO{1/1,2/0,1/2,2/1,3/3,3/0,2/2,3/1,2/3,3/2,2/4,4/2,3/4,4/3}
\end{scope}

\begin{scope}[xshift=5cm,yshift=-4cm]
  \Axes{-.35}{4.35}{-.55}{4.35}
  \Xticks{1,2,3,4}
  \Yticks{1,2,3,4}
  \node[paneltitle] at (1.5,4.92) {Step $6$};

  \draw[orth] (1,2) -- (2,0) node[orthlabel, pos=.47, above left=1pt] {};
  \draw[orth] (2,4) -- (3,0) node[orthlabel, pos=.43, above right=0pt] {};
  \draw[orth] (3,4) -- (4,2) node[orthlabel, pos=.50, above right=0pt] {};

  \draw[orderguide] (1,0.66) -- (1,2.55);
  \draw[orderguide] (2,-.34) -- (2,4.55);
  \draw[orderguide] (3,-.34) -- (3,4.55);
  \draw[orderguide] (4,1.66) -- (4,3.55);

  \PlotO{1/1,2/0,1/2,2/1,3/3,3/0,2/2,3/1,2/3,3/2,2/4,4/2,3/4,4/3}
\end{scope}

\begin{scope}[xshift=9cm,yshift=-4cm]
  \Axes{-.35}{4.35}{-4.35}{4.35}
  \Xticks{1,2,3,4}
  \Yticks{-4,-3,-2,-1,1,2,3,4}
  \node[paneltitle] at (2,4.93) {Step $7\to8\to9$};

  \draw[orth] (1,2) -- (2,-2) node[orthlabel, pos=.46, right=1pt] {};

  \draw[orderguide] (1,-4.34) -- (1,2.55);
  \draw[orderguide] (2,-2.34) -- (2,4.55);

  \PlotG{3/0,3/1,2/3,3/2,3/3,4/2,3/4,4/3}
  \PlotO{1/-4,1/-3,1/-2,1/-1,1/0,1/1,1/2,
         2/-2,2/-1,2/0,2/1,2/2,2/3,2/4}
\end{scope}

\end{tikzpicture}%
    }
    \caption[Visualizing the proof of \cref{prop:mut_sing_7}.]{Visualizing the proof of \cref{prop:mut_sing_7}.  The coordinate $(a,b)$ represents the twist
    $\sheaf{\tY}(aH+bh)$, or $\sheaf{E}(aH+bh)$ for torsion sheaves.  Green circles denote line bundles,
    red crosses denote torsion sheaves, and purple stars mark coordinates where both occur.  
    Orange arrows denote
    the direction of movement of mutated objects.
    Green dashed segments indicate complete
    orthogonality and the green guides indicate the order of the semi-orthogonal decompositions.}
    \label{fig:cubic7-vis-mut-tikz}
\end{figure}

\subsection{The kernel of resolution}
\begin{lemma}[][lem:Xi_nodal_7]
  The left adjoint $\varXi\colon \widetilde{\Ku(Y)} \to \Dcatb(Z)$ to $\varTheta\colon \Dcatb(Z) \to \widetilde{\Ku(Y)}$ is given by 
  \[\varXi = s_* \circ \eta^* \circ (\sheaf{}(-1,1)[2] \otimes -)\circ \lmut{\cC'},\]
  where 
  \begin{equation}\label{equ:A2_mutation_C'}
      \cC' = \cC \otimes \sheaf{}(0,4) = \ord{\sheaf{}(1,0),\ \sheaf{}(1,1),\ \sheaf{}(1,2),\ \sheaf{}(1,3),\ \sheaf{}(1,4),\ \sheaf{}(2,2),\ \sheaf{}(2,3)}.
  \end{equation}
\end{lemma}
\begin{proof}
  Recall from \cref{prop:mut_sing_7} that 
  \begin{equation}
      \varTheta = (\sheaf{}(-2,0) \otimes -)\circ\rmut{\cC}\circ\eta_*\circ s^*.
  \end{equation}
  Since the Serre functor of $\Dcatb(\tY)$ is given by $\sfS = \sheaf{}(-2,-4)[7]$, by \cref{lem:adj_triple} we have the adjunction pair $\lmut{\sheaf{}(a,b)} \dashv \rmut{\sheaf{}(a-2,b-4)}$, viewed as endofunctors on $\Dcatb(\tY)$.
  Therefore the left adjoint of $\varTheta$ is given by 
  \begin{align*}
    \varXi &= s_! \circ \eta^* \circ \lmut{\cC \otimes \sheaf{}(2,4)}\circ(\sheaf{}(2,0) \otimes -) \\ 
    &= s_! \circ \eta^*\circ(\sheaf{}(2,0) \otimes -) \circ \lmut{\cC \otimes \sheaf{}(0,4)} \\ 
    &= s_! \circ \eta^*\circ(\sheaf{}(2,0) \otimes -) \circ \lmut{\cC'}
  \end{align*}
  where $s_!$ is the left adjoint of $s^*$. By Grothendieck--Verdier duality, $s_! = s_*\circ(\omega_{D/Z}[2] \otimes -)$. Since $D$ is a $\pro^2$-bundle over $Z$ given by the bundle $\EEE|_{Z}$, we have 
  \[ \omega_{D/Z} = s^*\det \EEE^\dual \otimes \sheaf{\pro_Z(\EEE)}(-3) = \eta^*\sheaf{}(0,7) \otimes \eta^*\sheaf{}(-1,-2)^{\otimes 3} = \eta^*\sheaf{}(-3,1).\]
  This gives the claimed form for $\varXi$.
\end{proof}

\begin{lemma}[][lem:cubic7_sing_ker_compute]
    Let $\varGamma \coloneqq D \cap E \cong Z \times \ell$. Denote by $u\colon \varGamma \hookrightarrow E$ and $d\colon \varGamma \hookrightarrow D$ the corresponding inclusions. Let $v\coloneqq s \circ d\colon \varGamma \to Z$. Then the kernel of the categorical resolution $\sigma_* \circ \varTheta\colon \Dcatb(Z) \to \Ku(Y)$ is given by 
    \begin{equation}
        \ker(\sigma_*\circ\varTheta) = \ord{v_*\ab(u^*\varPhi\Dcatb(\ell,\sB_0) \otimes \sheaf{\varGamma}(-H+h)[2])}_{\mathrm{thick}}.
    \end{equation}
\end{lemma}
\begin{proof}
    By \cref{prop:sing_cubic7_ker_abstract}, the kernel of $\sigma_*|_{\widetilde{\Ku(Y)}}\colon \widetilde{\Ku(Y)} \to \Ku(Y)$ is given by $\iota_*\Phi\Dcatb(\ell,\sB_0)$. Hence 
    \begin{equation}
    \ker(\sigma_*\circ\varTheta)
    =
    \varTheta^{-1}\ab(\ker(\sigma_*|_{\widetilde{\Ku(Y)}}))
    =
    \ord{\varXi\ab(\iota_*\varPhi\Dcatb(\ell,\sB_0))}_{\mathrm{thick}}.
    \end{equation}
    We claim that $\lmut{\cC'}$ acts as identity on $\iota_*\varPhi\Dcatb(\ell,\sB_0)$. Indeed, for $a \in \Z$ and $0 \leq b \leq 4$, we have 
    \begin{align}
        \RHom(\sheaf{\tY}(aH+bh),\ \iota_*\varPhi\Dcatb(\ell,\sB_0))
        = \RHom(\sheaf{E}(aH+bh),\ \varPhi\Dcatb(\ell,\sB_0)) = 0,
    \end{align}
    where the last equality follows from the SOD \eqref{equ:quad_fib_7} for $\Dcatb(E)$. With $\cC'$ given by \eqref{equ:A2_mutation_C'}, the orthogonality given above implies that $\lmut{\cC'}\iota_*\varPhi\Dcatb(\ell,\sB_0) = \iota_*\varPhi\Dcatb(\ell,\sB_0)$. By \cref{lem:Xi_nodal_7} we have 
    \begin{equation}
        \varXi\ab(\iota_*\varPhi\Dcatb(\ell,\sB_0)) = s_*\eta^*\ab( \iota_*\varPhi\Dcatb(\ell,\sB_0) \otimes \sheaf{\tY}(-H+h)[2]).
    \end{equation}
    Consider the Cartesian diagram:
    \begin{equation}
        \begin{tikzcd}[ampersand replacement=\&,cramped]
        	E \& \tY \\
        	\varGamma \& D
        	\arrow["\iota", from=1-1, to=1-2]
        	\arrow["u", from=2-1, to=1-1]
        	\arrow["d", from=2-1, to=2-2]
        	\arrow["\eta"', from=2-2, to=1-2]
        \end{tikzcd}
    \end{equation}
    By \cite[Proposition 2.29]{catt23}, we have a derived base change formula $\eta^*\circ\iota_* = d_* \circ u^*$. By this base change and the projective formula, we have 
    \begin{align*}
        \varXi\ab(\iota_*\varPhi\Dcatb(\ell,\sB_0)) 
        &= s_*\ab(d_*u^*\varPhi\Dcatb(\ell,\sB_0) \otimes \eta^*\sheaf{\tY}(-H+h)[2]) \\ 
        &= s_*d_*\ab(u^*\varPhi\Dcatb(\ell,\sB_0) \otimes d^*\eta^*\sheaf{\tY}(-H+h)[2]) \\
        &= v_*\ab(u^*\varPhi\Dcatb(\ell,\sB_0) \otimes \sheaf{\varGamma}(-H+h)[2]). \qedhere
    \end{align*}
\end{proof}
Using the formula \cite[(23)]{kuzQuaFib} (in which $\cE' = \cE'_{-1,1}$ in the notation of \emph{loc.\ cit.}) for $\varPhi$, we have, for $A \in \Dcatb(\ell,\sB_0)$,
\begin{align}
    \varXi\iota_*\varPhi(A) &= v_*\ab(u^*\varPhi(A) \otimes \sheaf{\varGamma}(-H+h)[2]) \\
    &= v_*\ab(u^*p^*A \otimes_{u^*p^*\sB_0} u^*\cE'_{-1,1} \otimes_{\sheaf{\varGamma}} \sheaf{\varGamma}(-H+h)[2]) \\ 
    &= v_*\ab(u^*p^*A \otimes_{u^*p^*\sB_0} u^*(\cE'_{-1,1} \otimes_{\sheaf{E}} \sheaf{E/\ell}(1) \otimes_{\sheaf{E}} p^*\sheaf{\ell}(-1) )[2]) \\ 
    &= v_*\ab(q^*A \otimes_{q^*\sB_0} u^*\cE'_{0,-1})[2].
\end{align}
In the above, $q \coloneqq p \circ u\colon \varGamma \to \ell$ is a projection map, and the last equality follows from the periodicity \cite[Lemma 4.5]{kuzQuaFib}.

\begin{proof}[Proof of \cref{thm:sing}]
    It remains to show that the kernel computed in \cref{lem:cubic7_sing_ker_compute} is equivalent to $i^*\Ku(W) \subset \Dcatb(Z)$. For this we need a description of the smooth intersection $W = \bV(f_1,f_2)$ of two quadrics in $\pro^6$ and its Kuznetsov component from \cite[Corollary 5.7]{kuzQuaFib}:
    \begin{equation}
        \Dcatb(W) = \ord{ \varPhi^W_{0,0}\Dcatb(\ell,\sB_0),\ \sheaf{W}(1),\ \sheaf{W}(2),\ \sheaf{W}(3) },
    \end{equation}
    where the faithful functor $\varPhi^W_{0,0}\colon \Dcatb(\ell,\sB_0) \to \Dcatb(W)$, whose essential image is $\Ku(W) \otimes \sheaf{W}(1)$, is given explicitly by 
    \begin{equation}
        \varPhi^W_{0,0}(A) = v_{W,*}\ab(q_W^*A \otimes_{q_W^*\sB_0} u_W^*\cE'_{0,0}),
    \end{equation}
    where $u_W\colon \varGamma_W \hookrightarrow E$ is the inclusion of $\varGamma_W = W \times \ell$ into $E$; and similarly $v_W\colon \varGamma_W \to W$ and $q_W\colon W \to \ell$ are maps obtained by base changes from $Z$ to $W$. Twisting by $\sheaf{W}(1)$ and pulling back along $i\colon Z \hookrightarrow W$, we have 
    \begin{align}
        i^*(\varPhi^W_{0,0}(A) \otimes \sheaf{W}(-1)) &= v_{*}\ab(q^*A \otimes_{q^*\sB_0} u^*\cE'_{0,0}) \otimes \sheaf{Z}(-1) \\ 
        &= v_{*}\ab(q^*(A \otimes \sheaf{\ell}(1)) \otimes_{q^*\sB_0} u^*\cE'_{0,-1}) \otimes \sheaf{Z}(-1) \\ 
        &= \varXi\iota_*\varPhi(A \otimes \sheaf{\ell}(1)) \otimes \sheaf{Z}(-1).
    \end{align}
    In particular, we have the identification of the essential image of the functors:
    \[
        \ker(\sigma_*\circ\varTheta) = \ord{\varXi\iota_*\varPhi\Dcatb(\ell,\sB_0)}_{\mathrm{thick}} = \ord{i^*\Ku(W) \otimes \sheaf{Z}(1)}_{\mathrm{thick}} \subset \Dcatb(Z). \qedhere
    \]
\end{proof}

\printbibliography[heading=bibintoc]

\end{document}